\documentclass[11pt,a4paper,oneside]{amsart}
\usepackage[a4paper,margin=2.5cm]{geometry}
\usepackage[dutch,english]{babel}
\usepackage{graphicx,caption,subcaption}                         
\usepackage[pdftex,hyperfootnotes=false,colorlinks=true,linkcolor=blue,%
citecolor=purple,filecolor=magenta,urlcolor=cyan]{hyperref}
\usepackage{amsmath,amssymb,amscd,amsthm,amsfonts,anyfontsize,dsfont,%
enumerate,fix-cm,layout,lipsum,lpic,mathrsfs,mdwlist,pigpen,stmaryrd,%
tensor,thmtools,tikz,txfonts,xspace}
\usepackage[all]{xy}
\usepackage[oldstyle]{libertine}%
\usepackage[T1]{fontenc}
\usepackage[utf8]{inputenc}
\usepackage{csquotes}
\makeatletter
\renewcommand*\libertine@figurestyle{LF}
\makeatother
\usepackage[libertine,libaltvw,liby]{newtxmath}
\makeatletter
\renewcommand*\libertine@figurestyle{OsF}
\makeatother
\usepackage{mathtools}
\usepackage[noabbrev]{cleveref}
\expandafter\def\csname ver@etex.sty\endcsname{3000/12/31}

\usepackage{autonum}

\theoremstyle{plain}                          
\newtheorem{theorem}{Theorem}
\newtheorem{proposition}[theorem]{Proposition}    
\newtheorem{lemma}[theorem]{Lemma}

\newtheorem{conjecture}[theorem]{Conjecture} 
\crefname{conjecture}{conjecture}{conjectures }      
\theoremstyle{definition}
\newtheorem{definition}[theorem]{Definition}
\newtheorem{prop-defin}[theorem]{Proposition-definition} 

\theoremstyle{remark}
\newtheorem{remark}[theorem]{Remark}

\DeclarePairedDelimiter\floor{\lfloor}{\rfloor}

\newcommand{\dd}{\mathrm{d}}

\newcommand{\mb}[1]{\mathbb{#1}} 

\newcommand{\mc}[1]{\mathcal{#1}}

\newcommand{\C}{\mb{C}}

\newcommand{\del}{\partial}

\def\oM{\overline{\mathcal{M}}}

\DeclareMathOperator*{\Res}{Res}

\begingroup\expandafter\expandafter\expandafter\endgroup
\expandafter\ifx\csname pdfsuppresswarningpagegroup\endcsname\relax
\else
\fi
{}

\begin{document}
%\title{Topological recursion for \( 2\)-spin Hurwitz numbers}
\title{Special cases of the orbifold version of Zvonkine's $r$-ELSV formula}
\author[G.~Borot]{Ga\"etan Borot}
\address{G.~B.: Max Planck Institut f\"ur Mathematik, Vivatsgasse 7, 53111 Bonn, Germany.}
\email{gborot@mpim-bonn.mpg.de}
\author[R.~Kramer]{Reinier Kramer}
\address{R.~K.: Korteweg-de Vries Instituut voor Wiskunde, Universiteit van Amsterdam, P.O. Box 94248, 
1090 GE Amsterdam, Netherlands.}
\email{r.kramer@uva.nl}
\author[D.~Lewanski]{Danilo Lewanski}
\address{D.~L.: Korteweg-de Vries Instituut voor Wiskunde, Universiteit van Amsterdam, P.O. Box 94248, 
1090 GE Amsterdam, Netherlands.}
\email{d.lewanski@uva.nl}
\author[A.~Popolitov]{Alexandr Popolitov}
\address{A.~P.: Korteweg-de Vries Instituut voor Wiskunde, Universiteit van Amsterdam, P.O. Box 94248, 
1090 GE Amsterdam, Netherlands.}
\email{a.popolitov@uva.nl}
\author[S.~Shadrin]{Sergey Shadrin}
\address{S.~S.: Korteweg-de Vries Instituut voor Wiskunde, Universiteit van Amsterdam, P.O. Box 94248, 
1090 GE Amsterdam, Netherlands.}
\email{s.shadrin@uva.nl}
\thanks{}

\begin{abstract}
	%We rewrite the cut-and-join relations for $r$-spin Hurwitz numbers in terms of correlators. We analyse these equations in the case $r = 2$ using the polynomiality proved in \cite{KLPS17}, and prove that the correlators satisfy the topological recursion of \cite{EyOr08} for the (already known) spectral curve $x = ye^{-y^2}$. This result together with \cite{KLPS17} proves the ELSV-formula for $2$-spin Hurwitz numbers. Our method does not obviously extend to general $r$.
	We prove the orbifold version of Zvonkine's $r$-ELSV formula in two special cases: the case of $r=2$ (completed $3$-cycles) for any genus $g\geq 0$ and the case of any $r\geq 1$ for genus $g=0$.  
\end{abstract}

\maketitle

\tableofcontents

% % % % % % % % % % % % % % % %
% % % % % % % % % % % % % % % %
% % % % % % % % % % % % % % % %
% % % % % % % % % % % % % % % %
% % % % % % % % % % % % % % % %
% % % % % % % % % % % % % % % %
 
% % % % % % % % % % % % % % % %
% % % % % % % % % % % % % % % %
% % % % % % % % % % % % % % % %
% % % % % % % % % % % % % % % %
% % % % % % % % % % % % % % % %
% % % % % % % % % % % % % % % %

\section{Introduction}

In this paper we consider the $q$-orbifold $r$-spin Hurwitz numbers. These are weighted numbers of coverings of the Riemann sphere by (possibly disconnected) Riemann surfaces of arithmetic genus \( g\), with
\begin{itemize}
\item[$\bullet$] one branchpoint with arbitrary ramification, at which the orders of ramification are given by a vector $\mu$ of length $n \coloneqq \ell(\mu)$,
\item[$\bullet$] one branchpoint where the orders of ramification are $q,q,\dots,q$ -- we assume that $|\mu|\coloneq\sum_{i=1}^n \mu_i$ is divisible by $q$,
\item[$\bullet$] and $b := \tfrac{(2g-2+n)q+|\mu|}{qr}$ other branchpoints whose ramification profile are given by the so-called completed $(r+1)$-cycles -- we assume that $b$ is integer.
\end{itemize}
The source curve in this case can be disconnected, and we denote these numbers by $h_{g;\mu}^{\bullet, q,r}$. If we assume in addition that the source curve is connected, then we denote the resulting Hurwitz numbers by $h_{g;\mu}^{\circ, q,r}$.

There are several different sources of interest in these numbers. Completed cycles naturally emerge in the respresentation theory of the symmetric group~\cite{KerovOlshanski1994} and in the relative Gromov-Witten theory of the projective line~\cite{OkPa06a,OP}. For us the most convenient way to define these numbers combinatorially is using the semi-infinite wedge formalism, see~\cite{KLPS17}. 

The numbers $h_{g;\mu}^{\circ, q,r}$ for $q=1$ are conjecturally related to the intersection theory of the moduli spaces of $r$-spin structures~\cite{Zvonkine2006}. This conjecture is generalized to an arbitrary $q\geq 1$ in~\cite{KLPS17}. This conjecture incorporates, as special cases, the ELSV-formula for simple Hurwitz numbers~\cite{ELSV01} (the case $q=r=1$ for us) and the Johnson-Pandharipande-Tseng formula for the orbifold Hurwitz numbers~\cite{JohnsonPandharipandeTseng} (the case $r=1$, arbitrary $q\geq 1$). Let us recall it.

\begin{conjecture}[Zvonkine's $qr$-ELSV formula]\label{conjELSV} We have:
	\begin{equation}
		h_{g;\mu}^{\circ, q,r} = \prod_{j=1}^n \frac{\big(\frac{\mu_j}{qr}\big)^{\big\lfloor\frac{\mu_j}{qr}\big\rfloor }}{\big\lfloor\frac{\mu_j}{qr}\big\rfloor !}
	\times
	\frac{(qr)^{2g-2+n+\frac{(2g-2+n)q+|\mu |}{qr}}}{q^{2g-2+n}}\times
	\int_{\oM_{g,n}} \frac{\mathrm{C}_{g,\mu}^{q,r}}{\prod_{j=1}^n (1-\frac{\mu_i}{qr}\psi_i)}.
	\end{equation} 
\end{conjecture}
Here we denote by $\floor*{a}$ the integral part of $a\in\mathbb{Q}$. The classes $\mathrm{C}_{g,\vec\mu}^{q,r}$ are certain classes on the moduli spaces of curves related to the twisted rational powers of the dualizing sheaf on the universal curve computed explicitly by Chiodo~\cite{Chio08}. We do not need their definition in this paper, and we refer to~\cite{Lewanski2016,KLPS17} for their description.

While the original motivation of Zvonkine in~\cite{Zvonkine2006} was related to the geometry of meromorphic differentials on curves, it has appeared to be a natural statement in a completely different context. Namely, this type of formulas emerge naturally in the context of spectral curve topological recursion developed by Chekhov, Eynard, and Orantin~\cite{ChEy06,EyOr07}. This conjecture is equivalent to the following one:
\begin{conjecture}\label{conjTR} The formal symmetric $n$-differentials 
	\begin{equation}\label{eq:definition-omega}
	d_1\otimes\cdots\otimes d_n\sum_{\mu_1, \dotsc, \mu_n =1}^\infty h_{g;\mu}^{\circ, q,r} \prod_{i=1}^n x_i^{\mu_i}, \qquad g\geq 0,\ n\geq 1,
	\end{equation} 
	are expansions in $x_1,\dotsc,x_n$ of the symmetric $n$-differentials $\omega_{g,n}(z_1,\dots,z_n)$ that are defined on
	the spectral curve given by $x(z)\coloneq  ze^{ - z^{qr}}$, $y(z)\coloneq z^q$ and satisfy the topological recursion on it. 
\end{conjecture}

The equivalence of these two conjectures is proved in~\cite{Eyna11} for $q=r=1$, in~\cite{SSZ} for $q=1$, arbitrary $r\geq 1$, and it follows from the result of~\cite{Lewanski2016} in general case. In fact, it is a very general statement that the topological recursion produces the integrals over the moduli spaces of curves of this type, see~\cite{Eyna11a,DOSS14}. This second conjecture was first proposed in the case $q=r=1$ in~\cite{BouchardMarino}. The computations supporting this generalization are available in~\cite{MSS},~\cite{Lewanski2016}, and~\cite{KLPS17}.

If we want to use the equivalence of these two conjectures in order to prove the ELSV-type formulas proposed in the first conjecture, we have to prove the second conjecture independently. As of writing, independent proofs of the second conjecture are known in the case $q=r=1$~\cite{DKOSS13} and $r=1$, arbitrary $q\geq 1$~\cite{DLPS}. A key property, the proof of the combinatorial structure dictated by the corresponding ELSV-type formula (see \cref{KeyP}), was also proved differently and independently of \cref{conjELSV,conjTR} in~\cite{KLS}.

In this paper we prove the second conjecture in two new series of cases, namely
\begin{itemize}
	\item[$\bullet$] for $r=2$, and arbitrary $q\geq 1$, $g\geq 0$ (\cref{thm:r3});
	\item[$\bullet$] for $g=0$, and arbitrary $q,r\geq 1$ (\cref{thm:g0tr}).
\end{itemize}
Thus we fully prove \cref{conjELSV} in genus $0$, and also in any genus for the completed $3$-cycles. 

Let us discuss the strategy of the proof. We take the approach to the topological recursion proposed in~\cite{BEO13,BoSh15}. It is proved in~\cite{KLPS17} that the formal power series in $x_1,\dotsc,x_n$ in \cref{eq:definition-omega} is the expansion of a symmetric $n$-differential form defined on the spectral curve identified from the case $(g,n) = (0,1)$. Then the topological recursion is equivalent to the following three properties of these symmetric differentials:
the projection property, the linear loop equation, and the quadratic loop equation~\cite[theorem 2.1]{BoSh15}. The projection property and the linear loop equation are also proved in~\cite{KLPS17}. Thus, \cref{conjTR} is reduced to the quadratic loop equation.

The quadratic loop equation is, therefore, the main problem that we address in this paper. Let us briefly recall it in a convenient form. Consider the function $x=z e^{- z^{qr}}$. It has $qr$ branch points $\rho_1,\dots,\rho_{qr}$. We choose one of them, $\rho_i$. Denote by $\sigma_i$ the corresponding deck transformation. For any function $f(z)$ we define its local skew-symmetrization $\Delta_i(f)(z)\coloneq f(z)-f(\sigma_i(z))$. Then the quadratic loop equation is equivalent to the property that 
\begin{equation}
\left.\Delta_i'\Delta_i''\left(
\frac{
	\omega_{g-1,n+2}(z',z'',z_{[n]})+\sum_{\substack{g_1+g_2=g \\ I_1\sqcup I_2 = [n]}} \omega_{g_1,|I_1|+1}(z',z_{I_1})
	\omega_{g_2,|I_2|+1}(z',z_{I_2})
}{dx(z')dx(z'') \prod_{i = 1}^n dx(z_i)}
\right) \right|_{z'=z''=z}
\end{equation}
is holomorphic in $z$ near the point $p_i$. Here by $\Delta'_i$ and $\Delta''_i$ we mean the operator $\Delta_i$ acting on the variables $z'$ and $z''$ respectively. By $[n]$ we denote the set $\{1,\dots,n\}$ (and we use this notation throughout the paper). By $z_I$, $I\subset [n]$, we denote the set of variables with indices in $I$, for instance, $z_{[n]}\coloneq \{z_1,\dots,z_n\}$. This property should be satisfied for any $i=1,\dots,qr$ and for any $g\geq 0$, $n\geq 0$. 

In order to prove the quadratic loop equation we use the cut-and-join equation for the completed $(r+1)$-cycles derived in~\cite{Rossi,SSZ12,Alexandrov}. We rewrite the cut-and-join equation as an equation for the $n$-point functions $H_{g,n}(x_{[n]})\sim \sum_{\ell(\vec\mu)=n} h_{g;\vec{\mu}}^{\circ, q,r} \prod_{i=1}^n x_i^{\mu_i}$ (\cref{cutresummed}). In the special cases of completed $3$-cycles (\( r=2\)) and genus $0$ (any $r\geq 1$) this equation has a particularly nice form that allows us to derive the quadratic loop equation using the symmetrization of this equation in one variable. 

\subsection{Organization of the paper}
In \cref{sec:definition-CJ} we recall the definition of the $q$-orbifold $r$-spin Hurwitz numbers. In \cref{corandcj}, we derive the cut-and-join equation, and give explicit formulas for the genus $0$ with $\ell(\mu) \in \{1,2\}$ and genus $1$ with $\ell(\mu) = 1$. In \cref{sec:Completed3Cycle} we revisit the computation of the previous section in the particular case of $r=2$ (completed $3$-cycles). In \cref{sec:2-spinTR} we prove \cref{conjTR}, and, therefore, \cref{conjELSV} for $r=2$. In \cref{sec:genus0} we prove \cref{conjTR}, and, therefore, \cref{conjELSV} for $g=0$. 

\subsection{Acknowledgments} R.~K., D.~L., A.~P., and S.~S. were supported by the Netherlands Organization for Scientific Research. G.B. is supported by the Max Planck Gesellschaft, and thanks the University of Amsterdam for hospitality and support during this project. A.P. is also partially supported by RFBR grant 16-01-00291. We thank Dimitri Zvonkine for very useful discussions. In particular, though it has never been written down, he has developed alternative approaches to the proof of his conjecture in the special cases that we consider in this paper.

\section{Spin orbifold Hurwitz numbers}\label{sec:definition-CJ}

We define the Hurwitz numbers under consideration \textit{via} the semi-infinite wedge formalism. For more on this formulation, see~\cite{KLPS17}.
\begin{definition}
The \emph{disconnected \( r\)-spin \( q\)-orbifold numbers \( h_{g;\mu}^{\bullet,r,q} \)} are defined in the infinite wedge formalism by
\begin{equation}
h_{g;\mu}^{\bullet,r,q} \coloneqq \bigg\langle \frac{\alpha_q^{|\mu|/q}}{q^{|\mu |/q}(|\mu|/q)!} \frac{\mc{F}_{r+1}^b}{(r+1)^b} \prod_{i=1}^{\ell (\mu )} \frac{\alpha_{-\mu_i}}{\mu_i} \bigg\rangle\,,
\end{equation}
where the number of simple ramification points is given by
\begin{equation}
b = \frac{(2g - 2 + \ell(\mu))q + |\mu |}{qr}\,.
\end{equation}
We will reserve the symbol \( b\) for this quantity, sometimes inferring \( g\) from it. Here, $\mc{F}_{r + 1}$ is the operator of multplication by the $(r + 1)$-completed cycle, $(\alpha_{m})_{m \in \mathbb{Z}}$ the generators of the Heisenberg algebra acting on the semi-infinite wedge, and $\langle \cdots \rangle$ the sandwich between two vacuum states.
The \emph{connected \( r\)-spin \( q\)-orbifold Hurwitz numbers \( h_{g;\mu}^{\circ,r,q}\)} are defined from the disconnected ones \textit{via} the inclusion-exclusion principle.
\end{definition}

Let \( \mc{R} \) be the ring of formal power series in countably many formal commuting variables \( p_1, p_2, \dotsc \). If $\mu $ is a tuple of nonnegative integers $(\mu_1,\ldots,\mu_n)$, we denote $p_{\mu} = \prod_{i = 1}^n p_{\mu_i} \in \mathcal{R}$. Let \( \mc{R}_n \) be the subspace of \( \mc{R} \) consisting of formal power series with monomials of degree \( n\) in the \( p_m \). The operator $D$ defined by \( D(p_{\mu}) = |\mu |p_{\mu} \) is a derivation on \( \mc{R} \).\par 
\begin{definition}
We define the partition function of $r$-spin \( q\)-orbifold Hurwitz numbers
\begin{equation}
Z^{r,q} = \exp\Big(\sum_{\substack{g \geq 0 \\ n\geq 1}} \frac{1}{n!}\,G^{r,q}_{g,n}\Big)\,, \qquad \qquad \qquad G^{r,q}_{g,n} = \sum_{\mu_1,\ldots,\mu_n \geq 0} h_{g;\mu} ^{\circ,r,q} \frac{ \beta^b}{b!}\,p_{\mu}\,.
\end{equation}
Often, we will omit the superscripts \( r\) and \( q\). \par
Note that \( G_{g,n}^{r,q} \) is the homogeneous component of the sum in the exponent which is of degree \( n\) in the \( p\)'s and of degree \( 2g-2+n \) in a second degree, where \( \deg \beta = r\) and \( \deg p_\mu = -\frac{\mu}{r} \). On the other hand, \( Z^{r,q} \) is a formal power series.\par
\end{definition}

This partition function is characterized by the cut-and-join equation, as follows. We define the function
\begin{equation}
\zeta(z) = e^{z/2} - e^{-z/2}\,,
\end{equation}
and for bookkeeping we give the formula
\begin{equation}
\frac{1}{\zeta(z)} = \sum_{k \geq 0} \frac{(2^{1 - 2k} - 1)B_{2k}}{2k!}\,z^{2k - 1} = \frac{1}{z} - \frac{z}{24} + \frac{7z^3}{5760} + O(z^5)\,.
\end{equation}
involving the Bernoulli numbers $B_{2k}$. The cut-and-join operator $Q_r $ is defined by the generating series $Q(z) = \sum_{r \geq 1} Q_r\,z^r$ where
\begin{equation}
\label{cutjoinops}
Q(z) =  \frac{1}{\zeta(z)} \sum_{s \geq 1} \Big(\sum_{\substack{n \geq 1 \\ k_1 + \cdots + k_n = s}} \frac{1}{n!}  \prod_{i = 1}^n \frac{\zeta(k_iz)\,p_{k_i}}{k_i}\Big)\Big(\sum_{\substack{m \geq 1 \\ l_1 + \cdots + l_m = s}} \frac{1}{m!} \prod_{j = 1}^m \zeta(l_jz)\partial_{p_{l_j}}\Big)\,.
\end{equation}
It is well-known, see e.g. \cite[Theorem 5.5]{SSZ12}, that
\begin{equation}
\label{cutjoinr}\Big(\frac{1}{r!}\,\frac{\partial}{\partial \beta} - Q_{r+1}\Big)\,Z^{r,q} = 0\,.
\end{equation}

\section{Correlators and cut-and-join equations}
\label{corandcj}

\subsection{Correlators}

We describe an equivalent way to repackage $r$-spin \( q\)-orbifold Hurwitz numbers. Let $\mc{S}_n$ be the ring of symmetric analytic functions in $n$ variables, and consider the injective morphism of vector spaces $\Phi \colon \mc{R}_{n} \rightarrow \mc{S}_n$ which sends $p_{\mu}$ to the symmetric monomial
\begin{equation}
{\rm M}_{\mu}(x_1,\ldots,x_n) = \frac{1}{n!}\sum_{\sigma \in \mathfrak{S}_n} \prod_{i = 1}^n x_i^{\mu_{\sigma(i)}}\,.
\end{equation}
We denote $D_{x_i}$ the operator $x_i\partial_{x_i}$. It is consistent with the previous notation in the sense that
\begin{equation}
\forall f \in \mathcal{R}_{n},\qquad \Phi(Df) = \Big(\sum_{i = 1}^n D_{x_i}\Big) \Phi(f)\,.
\end{equation}
We introduce one more notation: if $I$ is a set, $J \subseteq I$ a subset, and $(x_i)_{i \in I}$ is a tuple of variables, $x_{J}$ stands for $(x_j)_{j \in J}$.

\begin{definition}
We define the correlators as
\begin{equation}
H^{r,q}_{g,n}(x_1,\ldots,x_n) = \Phi(G^{r,q}_{g,n})\big|_{\beta = 1}\,,
\end{equation}
\end{definition}
We would like to write down the cut-and-join equation \eqref{cutjoinr} solely in terms of the correlators. For this purpose we define operators $Q_{d;K_0,m}^{(k)}$ by a generating series
\begin{equation}
\label{Qrdef} \sum_{d \geq 0} Q_{d;K_0,m}^{(k)}\,z^{2d} = \frac{z}{\zeta(z)} \prod_{i \in K_0 \cup \{ k\} } \frac{\zeta(zD_{x_i})}{zD_{x_i}} \circ \prod_{j = 1}^m \frac{\zeta(zD_{\xi_j})}{z}\bigg|_{\xi_j = x_k}
\end{equation}
involving tuples of variables $(x_i)_{i \in K_0}$ and $m$ dummy variables denoted $(\xi_j)_{j = 1}^m$. The result of application of $Q_{d;K_0,m}^{(k)}$ to $F(x_{K_0},\xi_{1},\ldots,\xi_{m})$ only involves the variables $x_{K_0}$ and \( x_k \). In this definition, we stress that the operator $D_{x_k}$ in the first factor acts on the variable $x_k$ created by specialization of all $\xi_j$ to $x_k$. We have for instance
\begin{align}
Q_{0;K_0,m}^{(k)} & = \prod_{j = 1}^m D_{\xi_j}\bigg|_{\xi_j = x_k}\, ; \\ 
Q_{1;K_0,m}^{(k)} & = \frac{1}{24} \Bigg( D_{x_k}^2 \circ \bigg[\prod_{j = 1}^m D_{\xi_j}\Big|_{\xi_j = x_k}\bigg] + \Big(\sum_{i \in K_0} D_{x_i}^2 + \sum_{j=1}^m D_{\xi_j}^2  - 1\Big)\prod_{j = 1}^m D_{\xi_j}\bigg|_{\xi_j = x_k} \Bigg) \, .
\end{align}

\begin{proposition}
\label{cutjco}For any $g \geq 0$ and $n \geq 1$, we have
\begin{equation}
\label{cutandjfull}
\begin{split}
 \frac{B_{g,n}}{r !} H_{g,n}(x_{[n]})
& = \!\!\! \sum_{\{ k\} \sqcup \bigsqcup_{j=0}^\ell K_j = [n]} \! \frac{1}{l!} \! \sum_{\substack{m\geq 1,d \geq 0 \\ |K_0| + m + 2d = r+1}}\!\!\! \frac{1}{m!} \sum_{\substack{ \bigsqcup_{j=1}^\ell M_j = [m] \\ M_j \neq \emptyset}} \,\sum_{\substack{g_1,\ldots,g_{\ell} \geq 0 \\ g = \sum_{j = 1}^{\ell} g_j + m - \ell + d}}\\
& \qquad \qquad Q_{d;K_0,m}^{(k)} \bigg[\prod_{i \in K_0} \frac{x_i}{x_k - x_i} \prod_{j = 1}^{\ell} H_{g_j,|K_j| + |M_j|}(x_{K_j},\xi_{M_j})\bigg]\,,
\end{split}
\end{equation}
where \( B_{g,n} \coloneqq \frac{1}{r} \big(2g-2+n+ \frac{1}{q} \sum_{i=1}^n D_{x_i} \big) \).
\end{proposition}

The integer \( \sum_{j = 1}^{\ell} g_j + m - \ell \) is the genus of a surface obtained by glueing along boundaries a sphere with $m$ boundaries to a surface with $\ell$ connected components of respective genera $g_j$ and numbers of boundaries $|M_j|$, such that $\sum_j |M_j| = m$. Hence, $d$ can be interpreted as a \emph{genus defect}. When $g = 0$, we must have $\ell = m$, $g_j = 0$ for all $j$ and $d = 0$ in this equation, and it becomes a functional equation involving $H_{0,n'}$ only. For $(g,n) \neq (0,1)$, $H_{g,n}$ always appears in the right-hand side of \cref{cutandjfull} in the terms where $K_0 = \emptyset $, $\ell = m$, $d = 0$, $K_a = V \setminus \{ k\}$ for some $a \in [\ell ]$. They contribute to a term
\begin{equation}
\sum_{k = 1}^n \frac{\big( D_{x_k}H_{0,1}(x_k)\big)^r}{r!}\,D_{x_k}H_{g,v}(x_{[n]})
\end{equation}
For \( (g,n) = (0,1)\), the same term contributes, but it collapses to \( \frac{1}{r!} \big( D_{x_1}H_{0,1}(x_1)\big)^{r+1} \).

\begin{proof} We examine the homogeneous component of degree $n$ in the \( p\)'s and degree \( 2g-2+n-r\) in the grading where \( \deg \beta = r \) and \( \deg p_\mu = -\frac{\mu}{q} \) (effectively tracking the genus) in
\begin{equation}
\label{cutj} \frac{1}{r!}\,\frac{\partial}{\partial \beta}\,\ln Z = [z^{r+1}]\,Z^{-1}Q(z) Z\,. \end{equation}
The left-hand side of \cref{cutj} is
\begin{equation}
\frac{2g - 2 + n + \frac{1}{q}D}{r\cdot r!}\, \,.
\end{equation}
Applying $\Phi$ will replace $p_{\mu}$ by monomials $x_1^{\mu_1}\cdots x_n^{\mu_n}$. Let us consider the effect of the same operation in the right-hand side of \cref{cutj} before extracting the coefficient of $z^{r+1}$. A non-empty subset $L \subseteq [n]$ of the variables $(x_i)_{i = 1}^n$ will be used in the replacement of $\prod_i p_{k_i}$ from \( Q \). This will produce $\prod_{i \in L} x_i^{\mu_i}$ where $\mu$ is a permutation of $k$.\par
By a standard trick, \( e^{-F} (\prod \Delta_i) e^F = \prod (\Delta_i F)\) for differential operators \( \Delta_i \), so the other variables will appear by% substituting the $p$'s which remain in terms of the form
\begin{equation}
\prod_{j = 1}^{\ell} \Big(\prod_{i \in M_j} \partial_{p_{l_i}}\Big) G_{g_j,|K_j| + |M_j|}\,,
\end{equation}
where $(M_j)_{j = 1}^{\ell}$ is a partition of $[m]$ into non-empty subsets, $(K_j)_{j = 1}^{\ell}$ is a partition of $[n] \setminus L$ by possibly empty subsets, $(g_j)_{j = 1}^{\ell}$ is a sequence of nonnegative integers remembering the power of $\beta $ pulled out by the derivations acting on the exponential generating series $Z$. We have the constraint
\begin{equation}
\label{genusc} 2g - 2 + n -r = \sum_{j = 1}^{\ell} \big(2g_j - 2 + |K_j| + |M_j|\big)
\end{equation}
coming from identification of the exponent of $\beta $.

More precisely, the contribution of the variables in $[n] \setminus L$ will be of the form 
\begin{equation}
\oint \frac{\dd\xi}{2{\rm i}\pi \xi^{s + 1}} \bigg[\prod_{i = 1}^m \zeta(zD_{\xi_i}) \prod_{j = 1}^{\ell} H_{g_j,|K_j| + |M_j|}(x_{K_j},\xi_{M_j})\bigg]\bigg|_{\xi_a = \xi}\,,
\end{equation}
where $s = k_1 + \cdots + k_n$. The variables $x_L$ then contribute to
\begin{equation}
\label{cope}\Big[\prod_{i \in L} D_{x_i}^{-1}\zeta(zD_{x_i})\,x_i^{k_i}\Big] \oint \frac{\dd\xi}{2{\rm i}\pi \xi^{k_1 + \cdots + k_n + 1}} \bigg[\prod_{i = 1}^m \zeta(zD_{\xi_i}) \prod_{j = 1}^{\ell} H_{g_j,|K_j| + |M_j|}(x_{K_j},\xi_{M_j})\bigg]\bigg|_{\xi_a = \xi}\,.
\end{equation}
We should then perform the sum over positive $k$'s, using
\begin{equation}
\label{sumk}\sum_{k_1,\ldots,k_n \geq 1} \frac{x_1^{k_1}\cdots x_n^{k_n}}{\xi^{k_1 + \cdots + k_n + 1}} = \frac{1}{\xi} \prod_{i = 1}^n \frac{x_i}{\xi - x_i} = \frac{(-1)^{n}}{\xi} + \sum_{k = 1}^n \frac{1}{\xi - x_k} \prod_{i \neq k} \frac{x_i}{x_k - x_i}\,.
\end{equation}
The kind of computation we must do with this expression is a contour integral/extraction of coefficient
\begin{equation}
\sum_{k_1,\ldots,k_n \geq 1} \oint \frac{\dd\xi}{2{\rm i}\pi}\,\frac{x_1^{k_1}\cdots x_n^{k_n}}{\xi^{k_1 + \cdots + k_n + 1}}\,F(\xi)\,,
\end{equation}
where $F$ is some formal power series in $\xi$ without constant term. The term $\xi^{-1}$ in \cref{sumk} does not contribute and we find
\begin{equation}
\sum_{k_1,\ldots,k_n \geq 1} \oint \frac{\dd\xi}{2{\rm i}\pi}\,\frac{x_1^{k_1}\cdots x_n^{k_n}}{\xi^{k_1 + \cdots + k_n + 1}}\,F(\xi) = \sum_{k = 1}^n \Big[\prod_{i \neq k} \frac{x_i}{x_k - x_i}\Big] F(x_k)\,.
\end{equation}
We use this formula with the set of variables $(x_i)_{i \in L}$ rather that $(x_i)_{i = 1}^n$, and with
\begin{equation}
F(\xi) = \bigg[\prod_{j = 1}^m \zeta(zD_{\xi_j})\prod_{j = 1}^{\ell} H_{g_j,|K_j| + |M_j|}(x_{K_j},\xi_{M_j})\bigg]\bigg|_{\xi_a = \xi}\,.
\end{equation}
We should then apply the operator $\prod_{i \in K_0} D_{x_i}^{-1}\zeta(zD_{x_i})$ as it appeared in \cref{cope}, perform all necessary sums and finally extract the coefficient of $z^{r+1}$ to obtain $\Phi$ applied to the right-hand side. In this process, one has to carefully track the symmetry factors (there is a factor \( \frac{1}{l!} \) because the set partitions of \( [n] \) and \( [m] \) should be unordered but paired and a factor \( \frac{1}{m!} \) because all \( \xi \) are identical), and the outcome is
\begin{equation}
\label{combi}
\begin{split}
&[z^{r+1}] \sum_{m \geq 1} \frac{1}{m!} \sum_{\substack{L \sqcup \bigsqcup_{j=1}^\ell K_j = [n] \\ \bigsqcup_{j=1}^\ell M_j = [m] \\ L,M_1,\ldots,M_{\ell} \neq \emptyset \\ g_1,\ldots,g_{\ell} \geq 0}} \!\! \frac{1}{l!} \sum_{k \in L} \delta\Big(r + \sum_{j = 1}^{\ell} (2g_j - 2 + |K_j| + |M_j|) - (2g - 2 +n)\Big)\\
&\qquad \qquad \times \zeta^{-1}(z) \prod_{i \in L} D_{x_i}^{-1}\zeta(zD_{x_i}) \bigg[\prod_{j = 1}^m \zeta(zD_{\xi_j}) \prod_{i \in L \setminus \{k\}} \frac{x_{i}}{x_k - x_i} \prod_{j = 1}^{\ell} H_{g_j,|K_j| + |M_j|}(x_{K_j},\xi_{M_j})\bigg]\bigg|_{\xi_a = x_{i_0}}\,,
\end{split}
\end{equation}
where the Kronecker delta imposes the genus constraint \cref{genusc}. Since by definition $\sum_{j = 1}^{\ell} |M_j| = m$ and $\sum_{j = 1}^{\ell} |K_j| = n - |L|$, this genus constraint can be rewritten
\begin{equation}
g = \sum_{j= 1}^{\ell} g_j + (m - \ell) + d\,,
\end{equation}
with $d \geq 0$ defined by the formula \( m + |L| - 1 + 2d = r+1 \).\par
Let us rewrite \( L = K_0 \sqcup \{ k\} \), and notice that the operators $Q_{d;K_0,m}^{(k)}$ were precisely defined in \cref{Qrdef} so that
\begin{equation}
[z^{r+1}] \,\,\zeta^{-1}(z) \prod_{i \in K_0 \cup \{ k\}} D_{x_i}^{-1}\zeta(zD_{x_i})\Big[\prod_{j = 1}^m \zeta(zD_{\xi_j})\Big]\Big|_{\xi_j = x_k} = Q_{d;K_0,m}^{(k)}\,,
\end{equation}
and this puts \cref{combi} in the claimed form. 
\end{proof}

\subsection{Spectral curve: \texorpdfstring{$(g,n) = (0,1)$}{(g,n)=(0,1)} and \texorpdfstring{$(0,2)$}{(0,2)}.}

\begin{lemma}\label{H01}
Let $y(x) = D_{x}H_{0,1}(x)$. We have $x = y(x)^{\frac{1}{q}}e^{-y^{r}(x)}$.
\end{lemma}

Let $\mathcal{C}$ be the plane curve of equation $x = y^{\frac{1}{q}}e^{-y^{r}}$, it is a genus zero curve with maps
$$
y\,:\,\begin{array}{rcl} \mc{C} & \longrightarrow & \C \\ z & \longmapsto & z^q \end{array} \qquad {\rm  and}\qquad  x\,:\,\begin{array}{rcl} \mc{C} & \longrightarrow & \mathbb{C} \\  z & \longmapsto & z e^{-z^{qr}} \end{array}
$$
for the chosen global coordinate $z$. This last map has $qr $ simple ramification points. They are indexed by the $qr $-th roots of unity, here denoted $\omega^i $, and have coordinates
\begin{equation}
(x,y) = \big( (erq)^{-\frac{1}{rq}}\omega^i, \,(rq)^{-\frac{1}{r}} \omega^{qi} \big) \,.
\end{equation}
Their position in the $z$-coordinate is denoted
$$
\rho_i = (rq)^{-\frac{1}{rq}}\omega^i\,.
$$
Let $\sigma_i$ be the deck transformation of the branched cover $x\,:\,\mathcal{C} \rightarrow \mathbb{C}$ around $\rho_i$, and $\eta$ be a local coordinate such that $x(z) = x(\rho_i) + \eta^2$. We have:
\begin{lemma}
\label{Lkng}We have $\mathcal{S}_{z} y = 2(qr)^{-\frac{1}{r}} + \mc{O}(\eta^2)$. In particular, $\mathcal{S}_{z}y$ is locally invertible (for the multiplication) in a neighborhood of $\rho_i$.
\end{lemma}
\begin{proof} We indeed have $\mathcal{S}_{z}y = z^{q} + \sigma_i(z)^{q} = 2\rho_i^{q} + O(\eta^2)$. There are no odd powers in $\eta$ since $\mathcal{S}$ is the symmetrization operator.
\end{proof}
We introduce another coordinate $t$ such that
\begin{equation}
\frac{1}{t} = y^r - \frac{1}{qr} \,.
\end{equation}
The ramification points are all located at $t = \infty$, while $x \rightarrow 0$ corresponds to $t \rightarrow -qr $. For bookkeeping we give the formula
\begin{equation}
\label{Dtx} D_{x} = \frac{t^2(t + qr )}{qr }\,\partial_{t}\, .
\end{equation}
and if $\sigma_i$ denotes the deck transformation around the ramification point $\rho_i$ 

The following formula for $H_{0,2}$ was derived in \cite{KLPS17} via semi-infinite wedge formalism, we re-derive it here to test the cut-and-join equation and to demonstrate how to compute with it.

\begin{lemma}\label{H02}
We have
\begin{equation}
\label{H02exp} 
H_{0,2}(x_1,x_2) = \ln (z_1 - z_2) - \ln(x_1 - x_2) - y_1^r - y_2^r \,.
\end{equation}
In particular we obtain
\begin{equation}
\label{W02exp} W_{0,2}(x_1,x_2) \coloneqq D_{x_1}D_{x_2}H_{0,2}(x_1,x_2) = \frac{(Dz_1)(Dz_2)}{(z_1 - z_2)^2} - \frac{x_1x_2}{(x_1 - x_2)^2}\,,
\end{equation}
and
\begin{equation}
\label{W02xx}
W_{0,2}(x,x) = \frac{3t^4 + 4qrt^3 + (q^2r^2 - 1)t^2 + q^2r^2}{12r^2q^2}\,.
\end{equation}
\end{lemma}

%$H_{0,2}(x_1,x_2)$ is by definition a formal power series which is symmetric under exchange of its two variables $x_1,x_2$. Assuming $|x_1| < |x_2|$ we could write an analytic function -- the right-hand side of \cref{H02exp} -- whose power series expansion at $x_1 \rightarrow 0$ and then $x_2 \rightarrow 0$, is $H_{0,2}$. 
%In this assumption $x_1$ and $x_2$ do not play a symmetric role, therefore the asymmetry of the right-hand side of \cref{H02exp} does not contradict the symmetry of the formal power series $H_{0,2}$. Besides, the symmetry becomes manifest in \cref{W02exp} for the closely related generating series $W_{0,2}$.

\begin{proof}[Proof of \cref{H01}]
For $(g,n) = (0,1)$, the only terms in the right-hand side of the cut-and-join equation \eqref{cutandjfull} have $k = 1$ and genus defect $d = 0$, therefore the variable $x_{1} = x$ appears $m = r+1$ times, and we must have $\ell = m$, \emph{i.e.} $m$ factors of $D_{x}H_{0,1}$. One of the factors \( \frac{1}{(r+1)!} \) drops out against the sum over set partitions \( \bigsqcup_{j=1}^{r+1} M_j = [r+1]\), which is the sum over \( \mathfrak{S}_{r+1} \). We find
\begin{equation}
\frac{- H_{0,1} + \frac{1}{q}D_{x}H_{0,1}}{r \cdot r !} = \frac{(D_{x}H_{0,1})^{r+1}}{(r+1)!}\,.
\end{equation}
Let us define $y(x) = D_{x}H_{0,1}(x)$. Applying $\partial_{x}$, we get $-x^{-1}y + \frac{y'}{q} = r y'y^r$ and thus $ y'(\frac{y^{-1}}{q} - r y^{r - 1}) = x^{-1}$, which integrates to
\begin{equation}
\frac{1}{q} \ln y - y^r = \ln x + c
\end{equation}
for some constant $c$. Since $y(x) = x + O(x^2)$ when $x \rightarrow 0$, we must have $c = 0$, proving \cref{H01}.
\end{proof}

\begin{proof}[Proof of \cref{H02}]We denote $y_i = y(x_i)$ and $t_i = t(x_i)$ for $i \in \{1,2\}$. According to the remark below \cref{cutjco}, the right-hand side of the cut-and-join equation for $(g,n) = (0,2)$ contains $H_{0,2}$ as $(y_1^r D_{x_1} + y_2^r D_{x_2})H_{0,2}$. The remaining terms have genus defect $d = 0$ and correspond to $K_0 \neq \emptyset $. Now, $k$ takes the value $1$ or $2$, and $x_k$ appears $m = r $ times in $\ell = m$ functions $DH_{0,1} = y$. This leads to
\begin{equation}
\label{02cutj} \big(y_1^r - \tfrac{1}{qr}\big)D_{x_1}H_{0,2} + \big(y_2^r - \tfrac{1}{qr}\big)D_{x_2}H_{0,2} + \tfrac{y_1^r x_2 - y_2^r x_1}{x_1 - x_2} = 0\,.
\end{equation}
The solution we look for admits a formal power series expansion of the form
\begin{equation}
\label{sym02}H_{0,2}(x_1,x_2) = \sum_{k,l \geq 1} h_{k,l}\,x_1^k x_2^l,\qquad h_{k,l} = h_{l,k}\,.
\end{equation}
In particular we must have
\begin{equation}
\label{02cutc}\lim_{x_1 \rightarrow 0} H_{0,2}(x_1,x_2) = 0\,.
\end{equation}
One can check that
\begin{equation}
H_{0,2}^{(0)} \coloneqq \ln \Big( \frac{z_1 - z_2}{x_1 - x_2}\Big) - y_1^r - y_2^r 
\end{equation}
satisfies all these conditions. If $H_{0,2}^{(1)}$ is another solution, then $F = H_{0,2}^{(0)} - H_{0,2}^{(1)}$ must satisfy
\begin{equation}
\label{gensol} \Big(\big(y_1^r - \tfrac{1}{qr}\big)D_{x_1} + \big(y_2^r - \tfrac{1}{qr}\big)D_{x_2}\Big)F = 0\,.
\end{equation}
We remark that
\begin{equation}
\big(y_i^r - \frac{1}{qr}\big)D_{x_i} = \frac{t_i(t_i + qr)}{qr}\,\partial_{t_i} = -\partial_{u_i},\qquad u_i := \ln \Big(\frac{t_i + qr}{t_i}\Big) = \ln (qrz_i^{qr}) \,. 
\end{equation}
Therefore, the general solution of \cref{gensol} is $F = \varphi\big(z_1^{qr}z_2^{-qr}\big)$. Because \( x(z) = ze^{-z^{qr}} \) is locally invertible around \( x = z = 0\), this proves the only non-zero coefficients \( f_{k,l} \) are \( f_{k,-k} \). But because \( k \geq 1\) by \cref{sym02}, we get \( F=0\). This proves \cref{H02exp}. A simple computation leads to  \cref{W02exp,W02xx}.
\end{proof}

\subsection{Cut-and-join equation revisited}

We are going to transform the cut-and-join equation from \cref{cutjco} in order to treat the factors $\prod_{i \in K_0} D_{x_i}^{-1}\zeta(zD_{x_i})\,\frac{x_i}{x_k - x_i}$ at the same footing as correlator contributions. 
Let us define
\begin{equation}
\tilde{H}_{0,2}(x_1,x_2) = H_{0,2}(x_1,x_2) + H_{0,2}^{{\rm sing}}(x_1,x_2)\,,\qquad H_{0,2}^{{\rm sing}}(x_1,x_2) = \ln \Big(\frac{x_1 - x_2}{x_1 x_2}\Big)\,.
\end{equation}
Note that $\tilde{H}_{0,2}(\xi_1,\xi_2)|_{\xi_1 = \xi_2 = x}$ is not well-defined. When such an expression appears below, we adopt the convention that it should be replaced with $H_{0,2}(x,x)$, which is well-defined. Furthermore, for \( 2g-2+n > 0\), define \( \tilde{H}_{g,n}(x_{[n]}) \) by the following recursion:
\begin{equation}\label{cutresummed}
 \frac{B_{g,n}}{r!}\tilde{H}_{g,n}(x_{[n]})  = \!\!\!\! \sum_{\substack{m \geq 1, d \geq 0 \\ m + 2d = r+1}} \!\! \frac{1}{m!} \!\! \sum_{\substack{\{ k\} \sqcup \bigsqcup_{j=1}^\ell K_j = [n] \\ \bigsqcup_{j=1}^\ell M_j = [m] \\ M_j \neq \emptyset}}  \! \frac{1}{l!} \!\sum_{\substack{g_1,\ldots,g_{\ell} \geq 0 \\ g = \sum_j g_j + m - \ell + d}} \!\!\!\!\! Q_{d,\emptyset,m}^{(k)} \bigg[\prod_{j = 1}^{\ell} \tilde{H}_{g_j,|K_j| + |M_j|}(x_{K_j},\xi_{M_j})\bigg]\,.
\end{equation}
\begin{proposition}
\label{cutandjoinP} For $2g - 2 + n > 0$, the generating functions \( H_{g,n} \) and \( \tilde{H}_{g,n} \) are equal, unless \( 2g-2+n = r\), in which case they differ by an explicit constant.
\end{proposition}
\begin{remark}
As we are ultimately interested in the differentials \( d^{\otimes n}H_{g,n} \), these constants are of no real consequence for the remaining of the paper.
\end{remark}
\begin{proof}
We remark that
\begin{equation}
D_{x_i}^{-1}\zeta(zD_{x_i}) \frac{x_i}{x_k - x_i} = \ln \Big(\frac{x_k - e^{-z/2} x_i}{x_k - e^{z/2}x_i}\Big) = \ln \Big(\frac{e^{z/2}x_k  - x_i}{e^{z/2}x_k x_i} \frac{e^{-z/2}x_k x_i}{e^{-z/2}x_k  - x_i}\Big) = \zeta(zD_{x_k})\ln \Big(\frac{x_k - x_i}{x_kx_i }\Big)\,.
\end{equation}
Therefore, we can interpret the factors $D_{x_i}^{-1}\zeta(zD_{x_i})\,\frac{x_i}{x_k - x_i}$ in \cref{cutandjfull} as contributions of $H_{0,2}^{{\rm sing}}$. The sum $\tilde{H}_{0,2} = H_{0,2} + H_{0,2}^{{\rm sing}}$ is reconstructed in the left-hand side of  \cref{cutandjfull}, and treated in the same way as the other factors of $H$. Let us make the correspondence between the old and the new summation ranges. Now we are considering $m' = m + |K_0|$ the total number of occurences of $x_k$, $\ell' = \ell + |K_0|$ the total number of $H$-factors. These factors contain variables distinct from $k$, organized according to a partition $K_1' \sqcup \cdots \sqcup K'_{\ell'}$ where $K_{j}' = K_j$ for $1 \leq j \leq \ell $ and where $K_{\ell + j}$ for $1 \leq j \leq |K_0|$ are the singletons of elements of $K_0$. The genus attached to these $(0,2)$-factors is $g_{\ell + j} = 0$ for $1 \leq j \leq |K_0| $, and $m' - \ell' = m - \ell$, so the genus constraint keeps the same form
\begin{equation}
g = \sum_{j = 1}^{\ell'} g_j + m' - \ell' + d\,,
\end{equation}
while the genus defect is now defined by $m' - 1 + 2d = r$. The symmetry factors which occur in this resummation \textendash{} as the partitions of \( [n]\) and \( [m]\) can be reordered \textendash{} are accounted for in the formula given by \cref{cutresummed}, where we have removed all primes on the dummy indices of summations for an easier reading.\par
In relabelling we have added the term in the sum of the right-hand side where \emph{all} \( H_{g,n} \) are actually \( H_{0,2}^{\textrm{sing}} \). This corresponds to adding the term with \( m=\ell = 0\) in \cref{cutandjfull}. Unraveling the definitions, the extra term reads
\begin{align}
 \sum_{k=1}^n \sum_{d \geq 0}\delta_{n-1+2d,r+1} \delta_{g, d} Q_{d;[n]\setminus \{ k\},0}^{(k)} \bigg( \prod_{i \neq k} \frac{x_i}{x_k - x_i} \bigg) &= \sum_{k=1}^n \delta_{2g-2+n,r}[z^{2g}]\frac{z}{\zeta(z)} \prod_{i =1}^n \frac{\zeta(zD_{x_i})}{zD_{x_i}} \bigg( \prod_{i \neq k} \frac{x_i}{x_k - x_i} \bigg) \\
 &=\delta_{2g-2+n,r}[z^{2g}]\frac{z}{\zeta(z)} \prod_{i =1}^n \frac{\zeta(zD_{x_i})}{zD_{x_i}}\sum_{k=1}^n \prod_{i \neq k} \frac{x_i}{x_k - x_i}\,.
\end{align}
Now, the sum over \( k\) can be calculated \textit{via} residues:
\begin{equation}
\sum_{k=1}^n \prod_{i \neq k}  \frac{x_i}{x_k-x_i} = \sum_{k=1}^n \Res_{w=x_k} \frac{1}{w} \prod_{i=1}^n \frac{x_i}{w-x_i} = - \Res_{w=0} \frac{1}{w} \prod_{i=1}^n \frac{x_i}{w-x_i} = -1\,,
\end{equation}
using that the sum of residues of a meromorphic function is zero. As this is already constant, any derivatives in \( x\) give zero, so only the constant terms of the expansions $\frac{\zeta (zD_x)}{zD_x} = 1 + O(D_x)$ contribute. Hence the final contribution is
\begin{equation}
c_{g,n} \coloneqq -\delta_{2g-2+n,r}[z^{2g}]\frac{z}{\zeta(z)} = -\delta_{2g-2+n,r} \frac{(2^{1 - 2g} - 1)B_{2g}}{2g!}\,.
\end{equation}
Because of the Kronecker delta, \( \tilde{H}_{g,n} \) and \( H_{g,n} \) agree for \( 0 < 2g-2+n < r\). For \( 2g-2+n =0\), we get that
$$
\frac{B_{g,n}}{r!} \tilde{H}_{g,n}(x_{[n]}) -\frac{B_{g,n}}{r!} H_{g,n}(x_{[n]}) = c_{g,n}
$$
therefore
$$
\Big( 1+ \frac{1}{qr} \sum_{i=1}^n D_{x_i}\Big)\big( \tilde{H}_{g,n}(x_{[n]}) -H_{g,n}(x_{[n]})\big)= r! c_{g,n}\,.
$$
As both \( H_{g,n} \) and \( \tilde{H}_{g,n} \) are power series in the \( x_i \), their difference has to be \( r!c_{g,n} \).\par
If \( 2g-2+n > r\), the two functions are again equal, as we again have \( c_{g,n} =0\) by the Kronecker delta, and the different constants in \( \tilde{H}_{g',n'} \) on the right-hand side vanish by the differentiation included in the \( Q\)-operators.
\end{proof}

\subsection{Example: \texorpdfstring{$(g,n) = (1,1)$}{(g,n)=(1,1)}}

We show this computation as an illustration of the cut-and-join equation.

\begin{lemma}
We have $H_{1,1} = \tfrac{(qr + t)(1 - qt - t^2)}{24 q^2r y}$.
\end{lemma}

\begin{proof}
The cut-and-join equation for $(g,n) = (1,1)$ contains terms with genus defect $0$ and $1$. It reads
\begin{equation}
\label{eW11}\frac{y^r}{r!}\,D_{x}H_{1,1}(x) +  \frac{y^{r - 1}}{(r - 1)!}\,\frac{W_{0,2}(x,x)}{2} + \frac{1}{24}\Big(D_{x}^2 + \sum_{i = 1}^{r - 1} D_i^2 - 1\Big)\frac{y^{r - 1}}{(r - 1)!} = \frac{(q + D_{x})H_{1,1}(x)}{qr \cdot r!}\,,
\end{equation}
where $D_i^2$ is acting on the $i$-th factor in $y^{r - 1} = \prod_{i = 1}^{r - 1} y$. From \cref{W02xx} we know
\begin{align}
\frac{W_{0,2}(x,x)}{2(r-1)!} = \frac{3t^4 + 4qrt^3 + (q^2r^2 - 1)t^2 + q^2r^2}{24r^2q^2\cdot (r - 1)!}\,.
\end{align}
We also compute
\begin{equation}
\frac{1}{24 (r - 1)! }\Big(D_{x}^2 + \sum_{i = 1}^{r - 1} D_i^2 - 1\Big)= - \frac{2(r - 1)t^3 + qr(r - 1)t^2 + qr^2}{24qr\cdot r!}
\end{equation}
Substituting these expressions into \cref{eW11} and using \cref{Dtx}, we obtain:
\begin{equation}
\left( 1 - \frac{t(t+qr)}{q} \partial_t \right) H_{1,1} = \frac{y^{r-1} t^2}{24 q^2 }[3t^2 + 2q(r+1)t + (q^2 r -1)]
\end{equation}
Observing that $-t(t+qr)q^{-1}\partial_t = y \partial_y$, and imposing $H_{1,1} = y^{-1}F$ the equation becomes
\begin{equation}
\partial_y F = \frac{y^{r-1} t^2}{24 q^2 }[3t^2 + 2q(r+1)t + (q^2 r -1)]
\end{equation}
Applying back the inverse relation $\partial_y = -rt^2 y^{r-1} \partial_t$, we see that the solution which vanishes at $x = 0$ is obtained by
$$
F(t) = -\frac{1}{24 q^2 r} \int_{-qr}^{t} \big(3u^2 + 2q(r+1)u + (q^2 r -1)\big) \dd u. $$
Computing the integral yields the announced result.
\end{proof}

\section{Derivation of the cut-and-join equation for \texorpdfstring{$r + 1 = 3$}{r+1=3}} \label{sec:Completed3Cycle}
In this section, we will rederive the main result of the previous section, \cref{cutresummed} together with \cref{cutandjoinP}, for the specific case of \( r=2\). This is done both because the procedure is easier in this special case, and because we will make the formula more explicit. This more explicit form will be used to prove topological recursion in this case in \cref{sec:2-spinTR}. The present section can thus be read independently of \cref{corandcj}.  

Our starting point is again the cut-and-join equation \eqref{cutjoinr} (see \cite[Equation (32)]{SSZ12}, where we have an extra factor of \( r!\) because the weight of $\mathcal{F}_{r+1}$-operator is slightly different for us). Our goal is to derive from this equation an equation for the correlators
\begin{equation}
H_{g, n} (x_1, \dots, x_n) = \sum_{\mu_1, \dotsc, \mu_n \geq 1} \frac{h^{\circ, r,q}_{g, \mu}}{b!} \prod_{i=1}^n x_i^{\mu_i}
\end{equation}
directly for the case $r+1 = 3$, our main case of interest in this paper. In order to do so, we first derive the equation for the disconnected counterparts of $H_{g,n}$
\begin{equation}
H^\bullet_{g, n} (x_1, \dots, x_n) = \sum_{\mu_1, \dotsc, \mu_n \geq 1} \frac{h^{\bullet, r,q}_{g, \mu}}{b!} \prod_{i=1}^n x_i^{\mu_i}
\end{equation}
The generating functions $H^\bullet_{g,n}$ and $H_{g,n}$ are related by standard inclusion-exclusion formula. For example, in case of three points
\begin{align} \label{eq:3-point-disconnected-through-connected}
H^\bullet_{g,3}(x_1, x_2, x_3) = & H_{g,3}(x_1, x_2, x_3) \\
& + \sum_{g_1 + g_2 = g - 1} H_{g_1,2}(x_1, x_2) H_{g_2,1} (x_3) + H_{g_1,2}(x_2, x_3) H_{g_2,1} (x_1) + H_{g_1,2}(x_3, x_1) H_{g_2,1} (x_2)\\
& + \sum_{g_1 + g_2 + g_3 = g - 2} H_{g_1,1} (x_1) H_{g_2,1} (x_2) H_{g_3,1} (x_3)
\end{align}
Note that by the genus of each summand in the disconnected case, we understand its \emph{arithmetic} genus. Therefore, we have $\sum g_i = g - \#\big( \substack{\text{connected}\\\text{components}}\big) + 1$.
  
  %% While the cut-and-join equation \eqref{eq:cut-and-join-in-p} can be interpreted only formally
  %% (by considering coefficients in front of particular power of $\beta$ and particular monomial $p_{\vec\mu}$)
  %% \fixme{A remark about vectors and partitions},
  %% the equation for free energies $F_{g, n}$ makes sense on the whole spectral curve, thanks to the polynomiality
  %% statement of \cite{KLPS}.
  
 For the case $r + 1 = 3$ the cut-and-join operator from \cref{cutjoinops} is equal to (see \cite[p. 419]{SSZ12})
 \begin{align}
 Q_3 = & \tfrac{1}{6} \sum_{i,j,k \geq 1} \bigg( i j k p_{i + j + k} \frac{\del^3}{\del p_i \del p_j \del p_k} + (i + j + k) p_i p_j p_k \frac{\del}{\del p_{i + j + k}}\bigg) \\ \notag
& + \tfrac{1}{4} \sum_{\substack{i + j = k + l \\ i,j,k,l \geq 1}} \Big( i j p_k p_l \frac{\del^2}{\del p_i \del p_j}\Big) + \tfrac{1}{24} \sum_{i \geq 1} (2 i^3 - i) p_i \frac{\del}{\del p_i} \\ \notag
= & \tfrac{1}{6} Q_{p \del^3} + \tfrac{1}{6} Q_{p^3 \del} + \tfrac{1}{4} Q_{p^2 \del^2} + \tfrac{1}{24} Q_{p \del}\,,
\end{align}
where the last line introduces self-explanatory notation for the pieces of the cut-and-join operator with different number of multiplications and differentiations with respect to the variables $p_k$.

 \subsection{From \texorpdfstring{$p$}{p} to \texorpdfstring{$x$}{x}}
    
    %% ### I want somehow to explain that I'm able to consider different pieces of cut-and-join equation
    %% ### separately.
    %% \fixme{The explanation of how to do the substitution of $x_i$'s and why it works}

To derive the equation in $x$-variables we perform the following steps.
\begin{itemize}
\item We extract a coefficient $[\beta^{b-1} p_{\mu}]$ in front of a particular power $b-1$ of $\beta$ and a particular monomial $p_{\mu}$ from \cref{cutjoinr}. As a result we obtain something of the form
\begin{equation}
\text{LHS}_{b-1, \mu} = \text{RHS}_{b-1, \mu}
\end{equation}
\item Then we resum these individual equations in such a way that on the left-hand side we obtain one $H^\bullet_{g,n}$, with particular $g$ and $n$. It is clear that we need to take the following sum (note that we are summing over partitions here, not vectors, since all vectors $\mu$ differing by a permutation of components contribute to the same equation):\begin{equation}
\sum_{\mu_1\geq  \dotsb \geq \mu_n \geq 1} \text{LHS}_{b(g, \mu)-1, \mu} \sum_{\sigma \in \mathfrak{S}_n} x^{\mu_1}_{\sigma(1)} \dots x^{\mu_n}_{\sigma(n)} 
=\sum_{\mu_1, \dotsc, \mu_n \geq 1} \text{LHS}_{b(g, \mu)-1, \mu} x^{\mu_1}_1 \dots x^{\mu_n}_n
= \frac{B_{g,n}}{2!} H^\bullet_{g,n}\,,
\end{equation}
where $b(g, \mu) = \frac{2 g - 2 + n + |\mu|/q}{2}$ (the $-1$ in the power of \( \beta \)  accounts for $\partial_{\beta}$ in the equation). The operator \( B_{g,n} \coloneqq \frac{1}{2} \big( 2g-2+n +\frac{1}{q}\sum_{i = 1}^n D_{x_i}\big) \) reproduces the prefactor \( b\), which comes from the derivative.
\item Finally, we rewrite the right hand side, which now has the form
\begin{equation}
\sum_{\mu_1\geq \dotsb \geq \mu_n \geq 1} \text{RHS}_{b(g, \mu)-1, \mu} \sum_{\sigma \in \mathfrak{S}_n} x^{\mu_1}_{\sigma(1)} \dots x^{\mu_n}_{\sigma(n)}\,,
\end{equation}
as some differential operators acting on some $h^\bullet _{g,n}$s.
\end{itemize}

To perform the last step we analyse contributions of each $Q_{p^i \del^j}$
in turn. After that we group them in a smart way.
  
\subsubsection{The contribution of \texorpdfstring{$p \del^3$}{pd3}}
    
Let us consider the operator $Q_{p \del^3}$. The result of its action on the formal power series of the form
\begin{equation} \label{eq:generic-formal-power-series}
\sum_{\mu_1, \dots, \mu_n \geq 1} C_{\mu_1 \dots \mu_n} \frac{p_{\mu_1} \dots p_{\mu_n}}{n!}
\end{equation}
is (we shift $(n - 3) \rightarrow n$)
\begin{equation}
\sum_{i j k} \sum_{\mu_1, \dots, \mu_n \geq 1} i j k C_{i j k \mu_1 \dots \mu_n} \frac{p_{i+j+k} p_{\mu_1} \dots p_{\mu_n}}{n!}\,.
\end{equation}
We substitute the monomial $p_{i+j+k} p_{\mu_1} \dots p_{\mu_n}$ by
\begin{equation}
\sum_{\sigma \in \mathfrak{S}_{n+1}} x^{i+j+k}_{\sigma(1)} x^{\mu_1}_{\sigma(2)} \dots x^{\mu_n}_{\sigma(n+1)}\,.
\end{equation}
Each summand
\begin{equation}
i j k C_{i j k \mu_1 \dots \mu_n} x^{i+j+k}_{\sigma(1)} x^{\mu_1}_{\sigma(2)} \dots x^{\mu_n}_{\sigma(n+1)}\end{equation}
can be written as
\begin{equation}
\Big(D_{\xi_1} D_{\xi_2} D_{\xi_3} C_{i j k \mu_1 \dots \mu_n} \xi_1^i \xi_2^j \xi_3^k x^{\mu_1}_{\sigma(2)} \dots x^{\mu_n}_{\sigma(n+1)} \Big) \bigg{|}_{\xi_1 = \xi_2 = \xi_3 = x_{\sigma(1)}}\,,
\end{equation}
where \( D_\xi = \xi \partial_{\xi}\). Since for each value of $\sigma(1)$ there are $n!$ permutations from $\mathfrak{S}_{n+1}$ and their contributions are equal, because $C_{i j k, \mu}$ is symmetric in its indices, we see that the contribution of the $p \del^3$-term to the cut-and-join equation in terms of $x$ is equal to
\begin{equation}
\sum_{k = 1}^n \Big( D_{\xi_1} D_{\xi_2} D_{\xi_3} H^\bullet_{g - 2, n + 2} (\xi_1, \xi_2, \xi_3, x_{[n]\setminus \{ k\}} ) \Big) \bigg{|}_{\xi_1 = \xi_2 = \xi_3 = x_k}\,.
\end{equation}
A subtle point here is why we get precisely genus $g-2$. It is the result of direct counting. For every concrete $\mu$ we have, in case of $r + 1 = 3$, from the Riemann-Hurwitz formula for the left hand side
\begin{equation}
b  = g - 1 + \frac{n + |\mu|/q}{2}\,.
\end{equation}
On the other hand, for the contribution of $Q_{p \del^3}$ we can say that the number of completed cycles $b$ is one less, and the length of partition is bigger by two, while the size $|\mu|$ is the same. Therefore
\begin{equation}
b -1 = g_{p \del^3} - 1 + \frac{n + 2 + |\mu|/q}{2},
\end{equation}
i.e. $g_{p \del^3} = g - 2$.

\subsubsection{The contribution of \texorpdfstring{$p^2 \del^2$}{p2d2}}
The result of the action of $Q_{p^2 \del^2}$ on the formal power series of the form of \cref{eq:generic-formal-power-series} is
\begin{equation}
\sum_{\substack{i + j = k + l \\ i,j,k,l \geq 1}} \sum_{\mu_1, \dotsc, \mu_n \geq 1} i j C_{i j\mu} \frac{p_k p_l p_{\mu_1}\dots p_{\mu_n}}{n!}\,.
\end{equation}
After substitution of $p$ by $x$ it becomes
\begin{equation}
\sum_{a = 1}^\infty \sum_{\substack{i + j = a \\ i,j \geq 1}} \sum_{\substack{k + l = a \\ k,l \geq 1}} \sum_{\mu_1, \dotsc, \mu_n \geq 1} \sum_{\sigma \in \mathfrak{S}_{n + 2}} i j C_{i j\mu} \frac{x^k_{\sigma(1)} x^l_{\sigma(2)} x^{\mu_1}_{\sigma(3)} \dots x^{\mu_n}_{\sigma(n+2)}}{n!}\,.
\end{equation}
As we have the relation
\begin{equation} \label{eq:qanalog-2-sum}
\sum_{\substack{k + l = a \\ k,l \geq 1}} x^k y^l
= \frac{x^a y}{x - y} + \frac{y^a x}{y - x}\,,
\end{equation}
it is easy to see (the factor $n!$ again cancels with the number of permutations in $\mathfrak{S}_{n+2}$ with fixed $\sigma(1)$ and $\sigma(2)$, the extra $2$ comes from two summands in \cref{eq:qanalog-2-sum} that give equal contributions) that the contribution of $Q_{p^2 \del^2}$ is
\begin{equation}
2 \cdot \sum_{k \neq l} \frac{x_l}{x_k - x_l} \Big( D_{\xi_1} D_{\xi_2} H^\bullet_{g - 1, n} \big(\xi_1, \xi_2, x_{[n]\setminus \{ k,l\}}\big) \Big) \bigg{|}_{\xi_1 = \xi_2 = x_k}\,.
\end{equation}
The genus counting is analogous to the $p \del^3$-case.

\subsubsection{The contribution of \texorpdfstring{$p^3 \del$}{p3d}}
Quite analogously to the cases of $p \del^3$ and $p^2 \del^2$, the contribution of $Q_{p^3 \del}$ is equal to
\begin{equation}
3 \cdot \sum_{i \neq j \neq k} \Big( \frac{x_j}{(x_i - x_j)}\frac{x_k}{(x_i - x_k)} D_{\xi_1} H^\bullet_{g, n - 2} \big(\xi_1, x_{[n] \setminus \{ i, j, k\}} \big) \Big) \bigg{|}_{\xi_1 = x_i}\,.
\end{equation}
To derive it, one needs the following formula
\begin{equation}
\sum_{\substack{k + l + m = a \\ k,l,m \geq 1}} x^k y^l z^m = \frac{x^a y z}{(x - y)(x - z)} + \frac{y^a z x}{(y - z)(y - x)} + \frac{z^a x y}{(z - x)(z - y)}\,.
\end{equation}
The genus-counting is again straightforward.

\subsubsection{The contribution of \texorpdfstring{$p \del$}{pd}}
Finally, the contribution of $Q_{p \del}$ is
\begin{equation}
\sum_{k = 1}^n \Big( (2 D_{\xi_1}^3 - D_{\xi_1}) H^\bullet_{g - 1, n} (\xi_1, x_{[n]\setminus \{ k\} }) \Big) \bigg{|}_{\xi_1 = x_k}\,.
\end{equation}

\subsection{The unification}
Thus, we have obtained the following equation for the disconnected generating functions $H^\bullet_{g,n}$
\begin{eqnarray}
\tfrac{1}{2}B_{g,n} H^\bullet_{g,n}(x_{[n]})&= & 
\tfrac{1}{6} \sum_{k = 1}^n \Big( D_{\xi_1} D_{\xi_2} D_{\xi_3} H^\bullet_{g - 2, n + 2} \big( \xi_1, \xi_2, \xi_3, x_{[n] \setminus \{ k\}} \big)\Big) \Big{|}_{\xi_1 = \xi_2 = \xi_3 = x_k} \nonumber \\
&& + \tfrac{1}{4} \cdot 2 \sum_{k \neq l} \frac{x_l}{x_k - x_l} \Big( D_{\xi_1} D_{\xi_2} H^\bullet_{g - 1, n} \big( \xi_1, \xi_2, x_{[n] \setminus \{ k,l\} }\big)\Big) \Big{|}_{\xi_1 = \xi_2 = x_k} \nonumber \\
&& + \tfrac{1}{6} \cdot 3  \sum_{i \neq j \neq k} \frac{x_j}{(x_i - x_j)}\frac{x_k}{(x_i - x_k)} D_{x_i} H^\bullet_{g, n - 2} \big( x_{[n]\setminus \{ j,k\}} \big) \nonumber \\
\label{fsfs} && + \tfrac{1}{24}  \sum_{k = 1}^n (2 D_{x_k}^3 - D_{x_k}) H^\bullet_{g - 1, n}( x_{[n]}) \,.
\end{eqnarray}
Now define the \( m\)-disconnected, \( n\)-connected generating functions \( H_{g,m,n}(\xi_{[m]} \mid x_{[n]}) \) by keeping only those terms in the inclusion-exclusion formula where each factor contains at least one \( \xi \). For example, \( H_{g,1,n-1}(x_i \mid x_{[n]\setminus\{ i\}}) =H_{g,n}(x_{[n]}) \) and
\begin{equation}\label{inc-exc3}
\begin{aligned}
H_{g,3,n}(\xi_1, \xi_2, \xi_3\mid x_{[n]}) &=  H_{g-2,n+3}(\xi_1, \xi_2, \xi_3,x_{[n]}) \\
& \quad + \!\!\! \sum_{\substack{g_1 + g_2 = g - 1\\K_1 \sqcup K_2 = [n]}} \sum_{i = 1}^3 H_{g_1,1+|K_1|}(\xi_i,x_{K_1}) H_{g_2,2+|K_2|} (\xi_{[3]\setminus \{ i\}},x_{K_2}) \\
& \quad + \!\!\!\! \sum_{\substack{g_1 + g_2 + g_3 = g\\ K_1 \sqcup K_2 \sqcup K_3 = [n]}} \prod_{j=1}^3 H_{g_j,1+|K_j|} (\xi_1,x_{K_j})
\end{aligned}
\end{equation}
Then an easy inductive argument on the number of points $n$ shows that an equation very similar to \cref{fsfs} is true for these functions -- we just multiplied both sides by a factor of $2$.
\begin{align}
B_{g,n} H_{g,n}(x_{[n]}) &= 
\tfrac{1}{3} \sum_{k = 1}^n \Big( D_{\xi_1} D_{\xi_2} D_{\xi_3} H_{g - 2,3, n-1} \big( \xi_1, \xi_2, \xi_3 \mid x_{[n] \backslash \{ k\}} \big)\Big) \Big{|}_{\xi_1 = \xi_2 = \xi_3 = x_k} \\
& \quad +  \sum_{k \neq l} \frac{x_l}{x_k - x_l} \Big( D_{\xi_1} D_{\xi_2} H_{g - 1, 2,n-1} \big( \xi_1, \xi_2\mid x_{[n] \setminus \{ k,l\} }\big)\Big) \Big{|}_{\xi_1 = \xi_2 = x_k} \\
& \quad +  \sum_{i \neq j \neq k} \frac{x_i}{(x_k - x_i)}\frac{x_j}{(x_k - x_j)} D_{x_i} H_{g, n - 2} \big( x_{[n]\setminus \{ i,j\}} \big)\\
& \quad + \tfrac{1}{12} \sum_{k = 1}^n (2 D_{x_k}^3 - D_{x_k}) H_{g - 1, n}( x_{[n]}) \,.
\end{align}

Now we can unify the contributions of $Q_{p \del^3}$, $Q_{p^2 \del^2}$ and $Q_{p^3 \del}$ by changing the \( (0,2)\)-generating function to accomodate the rational factors in \( x\). First, we observe that the following equality holds:
\begin{equation}
D_\xi \ln\Big( \frac{\xi - x}{\xi x}\Big) = \frac{x}{\xi - x}\,.
\end{equation}

Suppose we substitute each $H_{0,2}(\xi, x)$ inside $H_{g,m,n}(\xi_{[m]} \mid x_{[n]})$ by the ``modified'' \( 2\)-point function $\tilde{H}_{0,2}(\xi, x)$, which is defined to be
\begin{equation}
\tilde{H}_{0,2}(\xi, x) = H_{0,2}(\xi, x) + \ln\Big( \frac{\xi - x}{\xi x}\Big)\,.
\end{equation}
We will denote these modified $H_{g,m,n}(\xi_{[m]}\mid x_{[n]})$ by $\tilde{H}_{g,m,n}(\xi_{[m]} \mid x_{[n]})$. Then the term
\begin{equation}
\tfrac{1}{3} \sum_{k = 1}^n \Big( D_{\xi_1} D_{\xi_2} D_{\xi_3} \tilde{H}_{g - 2, 3,n -1} \big(\xi_1, \xi_2, \xi_3| x_{[n]\setminus \{ k\}} \big) \Big) \Big{|}_{\xi_1 = \xi_2 = \xi_3 = x_k}
\end{equation}
contains contributions of $Q_{p \del^3}$, $Q_{p^2 \del^2}$ and $Q_{p^3 \del}$, corresponding to zero, one, or two $D_\xi$-operators acting on logarithmic corrections respectively. The genera match, because a factor \( \tilde{H}_{0,2} \) lowers the arithmetic genus of the product by one, and it is a direct check that the combinatorial coefficients match.

However, there is also a possibility that all three $D_\xi $-operators act on logarithmic corrections. This occurs only for $(g,n) = (0,4)$. By direct computation, the total contribution coming from this added possibility is equal to $-1$ (so it is constant in $x_i$s). Similarly, there is an extra contribution to the $p \del$-term coming from substitution of $H$ by $\tilde{H}$ in the case \( (g,n) = (1,2)\), but this is constant in $x$ as well -- it equals \( 1\). These extra terms only add this constant to \( H_{0,4} \) and \( H_{1,2}\), so they do not influence the recursion for other terms. Furthermore, they do not change the differentials \( \omega_{g,n} = d^{\otimes n} H_{g,n} \), which are the fundamental objects for topological recursion.\par
 So, defining
 \begin{equation}
 \tilde{H}_{g,n}(x_{[n]}) \coloneqq H_{g,n}(x_{[n]}) + \delta_{g,0}\delta_{n,2} \ln \Big( \frac{x_1-x_2}{x_1x_2}\Big) - \delta_{g,0}\delta_{n,4} +  \tfrac{1}{12} \delta_{g,1} \delta_{n,2}\,.
 \end{equation}
the cut-and-join equation in terms of $x$ can be written as
\begin{align}
B_{g,n} \tilde{H}_{g,n}(x_{[n]})
&=  \tfrac{1}{3} \sum_{k=1}^n \Big( D_{\xi_1} D_{\xi_2} D_{\xi_3} \tilde{H}_{g-2,3,n-1}(\xi_1,\xi_2,\xi_3 \mid x_{[n] \setminus \{ k\}}\Big) \Big{|}_{\xi_1 = \xi_2 = \xi_3 = x_k} \\
& \quad + \tfrac{1}{12} \sum_{i = k}^n (2 D_{x_k}^3 - D_{x_k}) \tilde{H}_{g - 1, n} ( x_{[n]})  \,.
\end{align}
This is the most concise version of the cut-and-join equation. However, for our purposes, it will be useful to have an even more explicit description. So we insert \cref{inc-exc3} into this equation, which yields (simplifying because we evaluate all \( \xi \)'s to the same value)
\begin{equation}
\label{explicitc-j2}
\begin{split}
B_{g,n} \tilde{H}_{g,n}(x_{[n]})
&=  \tfrac{1}{3} \sum_{k=1}^n \Big( D_{\xi_1} D_{\xi_2} D_{\xi_3} \tilde{H}_{g-2,n+2}(\xi_1,\xi_2,\xi_3, x_{[n] \setminus \{ k\}}\Big) \Big{|}_{\xi_1 = \xi_2 = \xi_3 = x_k} \\
& \quad + \!\! \sum_{\substack{g_1 + g_2 = g - 1\\ \{ k \} \sqcup K_1 \sqcup K_2 = [n]}} \!\!\!\! \Big( D_{x_k} H_{g_1,1+|K_1|}(x_k,x_{K_1})\Big) \Big( D_{\xi_1} D_{\xi_2} H_{g_2,2+|K_2|} (\xi_1 ,\xi_2 ,x_{K_2})\Big) \Big{|}_{ \xi_1 = \xi_2 = x_k} \\
& \quad + \tfrac{1}{3} \!\!\!\! \sum_{\substack{g_1 + g_2 + g_3 = g\\\{ k\} \sqcup K_1 \sqcup K_2 \sqcup K_3 = [n]}} \! \prod_{j=1}^3 \Big( D_{x_k}H_{g_j,1+|K_j|} (x_k,x_{K_j})\Big) \\
& \quad + \tfrac{1}{12} \sum_{k = 1}^n (2 D_{x_k}^3 - D_{x_k}) \tilde{H}_{g - 1, n} ( x_{[n]}) \,.
\end{split}
\end{equation}
This equation does indeed agree with \cref{cutresummed} for the case \( r = 2\).

\section{Topological recursion for Hurwitz numbers with \texorpdfstring{\( 3\)}{3}-completed cycles}
\label{sec:2-spinTR}

In this section, we show that the generating series for \( 2\)-spin \( q\)-orbifold Hurwitz numbers obey the topological recursion for the following curve derived in \cref{H01}, see also \cite{MSS,SSZ,SSZ12}
\begin{equation}
\mc{C} = \left\{\begin{array}{l}
x= ze^{-z^{2q}}\\
y = z^q \end{array}\right.
\end{equation}
We denote $(\rho_i)_{i = 1}^{qr}$ the (simple) ramification points of $x$ in $\mathcal{C}$, $\sigma_i(z)$ the deck transformation around $\rho_i$, and 
\begin{equation}
\label{SymASymOp}\Delta f(z) = f(z) - f(\sigma_i(z)),\qquad \mathcal{S} f(z) = f(z) + f(\sigma_i(z))
\end{equation}
the (skew)-symmetrization operator defined locally near $\rho_i$. The proof starts from \cref{explicitc-j2} and  %, which for \( r =2\) reads \fixme{Although there should be an extra factor \( \frac{1}{l!} \) there, as it is in BTR-higher-4. Added it here. See also \cite{SSZ12}}
%\begin{align}\label{eq:2-spinCJ}
%\begin{split}
%\big( 2g-2+n + \sum_{i=1}^n D_{x_i}\big) \tilde{H}_{g,n}(x_{[n]}) &=
%\frac{1}{6} \sum_{i=1}^n D_{\xi_1} D_{\xi_2} D_{\xi_3} \tilde{H}_{g-2,n+2} (\xi_1,\xi_2,\xi_3,x_{[n]\setminus \{ i\}} ) \Big|_{\xi_1 = \xi_2 =\xi_3 =x_i}\\
%&+ \frac{3}{2} \! \sum_{\substack{\{ i\} \sqcup K_1 \sqcup K_2 = [n] \\ g_1+g_2 = g-1}} \!\!\!\! D_{\xi_1} D_{\xi_2} \tilde{H}_{g_1,|K_1|+2}(\xi_1,\xi_2,x_{K_1})\Big|_{\xi_1=\xi_2=x_i} \!\! D_{x_i} \tilde{H}_{g_2,|K_2|+1}(x_i,x_{K_2})  \\
%&+ \sum_{\substack{\{i\} \sqcup K_1 \sqcup K_2 \sqcup K_3 = [n]\\ g_1+g_2+g_3 = g}} \prod_{j=1}^3 D_{x_i} \tilde{H}_{g_j,|K_j|+1}(x_i,x_{K_j})\\
%&+ \frac{1}{24} \sum_{i=1}^n (2D_{x_i}^3 -D_{x_i}) \tilde{H}_{g-1,n}(x_{[n]} )\,.
%\end{split}
%\end{align}

\begin{lemma} \cite{KLPS17}
\label{KeyP} For any $2g - 2 + n > 0$, the formal $n$-differential form $\omega_{g,n} \coloneqq d_1 \dots d_n H_{g,n}$ is the expansion at $x_1 = \ldots = x_n = 0$ of a meromorphic $n$-differential form on $\mathcal{C}^n$, which satisfies
\begin{itemize}
\item[$\bullet$] the linear loop equations: $(\dd x(z))^{-1}\mathcal{S}_{z} D_{x}\omega_{g,n}(z,z_{[n - 1]})$ is holomorphic when $z \rightarrow \rho_i$.
\item[$\bullet$] the projection property: $\omega_{g,n}(z,z_{[n - 1]}) = \sum_{i = 1}^{qr} \Res_{z' \rightarrow \rho_i} \Big(\int_{\rho_i}^{z'} \omega_{0,2}(\cdot,z)\Big)\omega_{g,n}(z',z_{[n - 1]})$. 
\end{itemize}
\end{lemma}

\begin{theorem}\label{thm:r3}
The differentials \( \omega_{g,n} \) satisfy topological recursion on \( \mc{C} \).
\end{theorem}
\begin{proof}
According to \cite{BoSh15}, it suffices to prove the quadratic loop equations, which are tantamount to saying that for any $2g - 2 + n \geq 0$ and $i \in [qr]$,
\begin{equation}\label{eq:Quale}
\Delta_{z} \Delta_{z'} D_{x} D_{x'} \tilde{H}_{g,2,n} \big( x(z),x(z') \mid x_{[n]}\big) \big|_{z' =z} 
\end{equation}
is holomorphic in \( z\) near $\rho_i$. Here
\begin{equation}
\tilde{H}_{g,2,n}(x,x' \mid x_{[n]}) = \tilde{H}_{g-1,n+2}(x,x',x_{[n]}) + \sum_{\substack{K_1 \sqcup K_2 = [n]\\ g_1+g_2 = g}} \tilde{H}_{g_1,|K_1|+1}(x,x_I) \tilde{H}_{g_2,|K_2|+1}(x',x_J)
\end{equation}
We will prove the quadratic loop equations in a way similar to \cite{DKOSS13}, fixing a ramification point $\rho_i$ for the remaining of the proof.\par
First, we apply \( \mc{S}_{z_1} \) to \cref{explicitc-j2}. By the linear loop equation, \( \mc{S}_z H_{g,n}(z,z_2, \dotsc, z_n) \) is holomorphic, and because \( x \) is invariant under the local involution \( \sigma_i \) by definition, \( D_x \) commutes with \( \mc{S}_z \). Hence, the left-hand side is holomorphic, just like the last term of the right-hand side and all terms in the \( k\)-sums except for \( k=1\).\par
On the other terms of the right-hand side, we use the elementary identity
\begin{equation}
\mc{S}_zf(\underbrace{z,\dotsc, z}_{r \text{ times}}) = 2^{1-r} \!\!\! \sum_{\substack{I \sqcup J = \llbracket r \rrbracket \\ |J|\,\,\text{even}}} \Big( \prod_{i \in I} \mc{S}_{z_i} \Big) \Big( \prod_{j \in J} \Delta_{z_j} \Big)f(z_1,\dotsc, z_r) \Big|_{z_i = z}\,,
\end{equation}
which in our case reduces to (using that all our choices for \( f\) will be invariant under the exchange of \( z \) and \( z'\))
\begin{equation}
\mc{S}_zf(z,z,z) =\tfrac{1}{4} \big( \mc{S}_z \mc{S}_{z'} \mc{S}_{z''} + 2\mc{S}_z\Delta_{z'} \Delta_{z''} %+ \Delta_z \mc{S}_{z'} \Delta_{z''} 
+ \Delta_z \Delta_{z'} \mc{S}_{z''} \big) f(z,z',z'') \big|_{z'=z''=z}\,.
\end{equation}
Again by the linear loop equations, the first term in the operator on the right-hand side results in holomorphic terms. Here we also used that the differentials, except the case \( (g,n) = (0,2)\), do not have diagonal poles. In this exceptional case, we only added a polar part if just one of the arguments was a \( \xi \), so we avoid the diagonal poles here as well.\par
To prove the quadratic loop equations, we use an induction on the Euler characteristic of the factors on which the \( \Delta \)-operators act: they either act on the same factor \( H_{g,n} \), in which case the Euler characteristic is given by \( -\chi = 2g-2+n\), or on separate factors, in which case the Euler characteristics of the factors must be added.\par
So consider the symmetrization of \cref{explicitc-j2} for \( (g,n)\) and assume the quadratic loop equations have been proved for all pairs \( (g',n')\) with \( 2g' - 2 +n' < 2g - 2+n\). We will split the equation into two parts. First consider the terms
\begin{align}
&\tfrac{1}{3} \big( 2\mc{S}_{z_1} \Delta_{z_1'}\Delta_{z_1''} + \Delta_{z_1} \Delta_{z_1'} \mc{S}_{z_1''} \big) D_{x_1} D_{x_1'} D_{x_1''} \tilde{H}_{g-2,n+2} (x_1',x_1'',x_{[n]} ) \big|_{x_i'=x_i''=x_i}\\
&\qquad + \! \sum_{\substack{K_1 \sqcup K_2 = [n] \setminus 1\\ g_1+g_2 = g-1}} \!\!\!\! 2\mc{S}_{z_1} \Delta_{z_1'} D_{x_1} D_{x_1'} \tilde{H}_{g_1,|K_1|+2}(x_1,x_1',x_{K_1})\big|_{x_1'=x_1} \!\! \Delta_{z_1}D_{x_1} \tilde{H}_{g_2,|K_2|+1}(x_1,x_{K_2}) \\
&= \mc{S}_{z_1} D_{x_1} \bigg( \Delta_{z_1'}\Delta_{z_1''} \Big( D_{x_1'} D_{x_1''} \tilde{H}_{g-2,n+2} (x_1',x_1'',x_{[n]} ) \\
&\qquad \qquad+ 2\!\! \sum_{\substack{K_1 \sqcup K_2 = [n] \setminus 1\\ g_1+g_2 = g-1}} \!\!\!\! D_{x_1'} \tilde{H}_{g_1,|K_1|+2}(x_1,x_1',x_{K_1}) D_{x_1''} \tilde{H}_{g_2,|K_2|+1}(x_1'',x_{K_2}) \Big)\bigg|_{x_1'=x_1''} \bigg) \bigg|_{x_1'=x_1}\,.
\end{align}
Now, for this last term, we use that
\begin{align}
&  \sum_{\substack{K_1 \sqcup K_2 = [n] \setminus 1\\ g_1+g_2 = g-1}} \!\!\!\! D_{x_1'} \tilde{H}_{g_1,|K_1|+2}(x_1,x_1',x_{K_1}) D_{x_1''} \tilde{H}_{g_2,|K_2|+1}(x_1'',x_{K_2}) \Big)\Big|_{x_1'=x_1''}  \\
&\qquad \qquad \qquad \qquad \qquad \qquad \qquad  = \tfrac{1}{2} \sum_{\substack{K_1 \sqcup K_2 = [n] \\ g_1+g_2 = g-1}} \!\!\!\! D_{x_1'} \tilde{H}_{g_1,|K_1|+1}(x_1',x_{K_1}) D_{x_1''} \tilde{H}_{g_2,|K_2|+1}(x_1'',x_{K_2}) \Big)\Big|_{x_1'=x_1''}\,,
\end{align}
because we have made of a choice of the set containing \( x_1\).\par
Hence, these terms together, before application of \( \mc{S}_{z_1} D_{x_1}|_{x_1'=x_1}  \), are the combination appearing in a quadratic loop equation, which is holomorphic by the induction hypothesis. Hence it is holomorphic after application of \( \mc{S}_{z_1} D_{x_1}|_{z_1'=x_1} \) as well. Again, we used that the differentials do not have diagonal poles.\par
The remaining terms are
\begin{align}
&  \Delta_{z_1} \Delta_{z_1'} \bigg( \sum_{\substack{K_1 \sqcup K_2 = [n] \setminus 1\\ g_1+g_2 = g-1}} \!\!\!\!\! D_{x_1} D_{x_1'} \tilde{H}_{g_1,|K_1|+2}(x_1,x_1',x_{K_1})\big|_{x_1'=x_1} \!\! \mc{S}_{z_1}D_{x_1} \tilde{H}_{g_2,|K_2|+1}(x_1,x_{K_2})\bigg)  \\
&+ \tfrac{1}{3} (2\mc{S}_{z_1} \Delta_{z_1'} \Delta_{z_1''} + \Delta_{z_1} \Delta_{z_1'} \mc{S}_{z_1''} ) \sum_{\substack{K_1 \sqcup K_2 \sqcup K_3 = [n] \setminus 1\\ g_1+g_2+g_3 = g}} \prod_{a=1}^3 D_{x_1} \tilde{H}_{g_a,|K_a|+1}(x_1,x_{K_a}) \,,
\end{align}
which can be written as
\begin{align}
&\sum_{\substack{K_1 \sqcup K_2 = [n] \setminus 1 \\ g_1+g_2 = g}} \mc{S}_{z_1} D_{x_1} \tilde{H}_{g_2,|K_2|+1}(x_1,x_{K_2} )\\[-1em]
&\qquad \qquad \cdot \Delta_{z_1} \Delta_{z_1'} D_{x_1} D_{x_1'} \Big( \tilde{H}_{g_1-1,|K_1|+2}(x_1,x_1',x_{K_1}) + \!\!\!\!\!\sum_{\substack{K_1' \sqcup K_1'' =K_1\\ g_1' + g_1'' = g_1}} \!\! \tilde{H}_{g_1',|K_1'|+1}(x_1,x_{K_1'}) \tilde{H}_{g_1'',|K_1''|+1}(x_1',x_{K_1''})\Big)\Big|_{x_1'=x_1}\,.
\end{align}
In this product, the first factor is holomorphic by the linear loop equations, while the second factor is exactly the combination appearing in the quadratic loop equation. By the induction hypothesis, the second factor is holomorphic as well, unless \( (g_1,K_1) = (g,[n]\setminus 1)\). Hence the only possibly non-holomorphic term in \cref{explicitc-j2} is
\begin{equation}
\Big( \mc{S}_{z_1} D_{x_1} \tilde{H}_{0,1}(x_1 )\Big) \cdot \Big( \Delta_{z_1} \Delta_{z_1'} D_{x_1} D_{x_1'} \tilde{H}_{g,2,n-1}(x_1,x_1' \mid x_{[n]}) \Big)\Big|_{x_1'=x_1}\,,
\end{equation}
which must therefore be holomorphic as well. We have $D_{x_1} \tilde{H}_{0,1} (x_1) = y_1$, and \cref{Lkng} guarantees that $\mathcal{S}_{z_1}D_{x_1}\tilde{H}_{0,1}(x_1)$ is invertible near $\rho_i$. This implies the quadratic  loop equations for \( (g,n)\).
\end{proof}

%%%%%%%%%%%%%%%%%%%%%%%%%%%%%%
%%%%%%%%%%%%%%%%%%%%%%%%%%%%%%
%%%%%%%%%%%%%%%%%%%%%%%%%%%%%%
%%%%%%%%%%%%%%%%%%%%%%%%%%%%%%
%%%%%%%%%%%%%%%%%%%%%%%%%%%%%%

\section{Topological recursion in genus zero} \label{sec:genus0}

In this section we prove that the genus-zero differentials $\omega_{g=0,n}= d^{\otimes n} H_{0,n}$ satisfy the  topological recursion relation for any integers $r$ and $q$. We do this in two steps. Firstly, we specialize the cut-and-join \cref{cutresummed} to genus zero. Secondly, we apply the symmetrizing operator to both sides of the equation and we analyze the holomorphicity of the terms in order to prove, by induction of the Euler characteristic, the quadratic loop equation of \cref{eq:Quale} in genus zero.

Let us now consider \cref{cutresummed}. For $g=0$, we have $g_j=0$ for every $j \in [l]$, the genus defect $d$ must also be zero, and $l=m=r+1$ should hold. This implies that the cardinality of every set $M_j$ must be equal to one. Therefore the choice of the sets $M_j$ is equivalent to the choice of a permutation of $r+1$ elements. By introducing these simplifications, the cut-and-join equation restricts to

\begin{equation}
 \frac{B_{0,n}}{r!}\tilde{H}_{0,n}(x_{[n]})  = \!\!  \frac{1}{(r+1)!(r+1)!} \!\!\!\!\!\!\! \sum_{\substack{ \sigma \in \mathfrak{S}_{r+1} \\ \{ k\} \sqcup \bigsqcup_{j=1}^{r+1} K_j = [n] }} \!\!\!\! \prod_{j = 1}^{r+1} D_{\xi_j} \bigg[\prod_{j = 1}^{r+1} \tilde{H}_{0,|K_j| + 1}(x_{K_j},\xi_{\sigma(j)})\bigg] \bigg|_{\xi_j = x_k}\,,
\end{equation}
where \( B_{0,n} \coloneqq \frac{1}{r} \big(n - 2+ \frac{1}{q}\sum_{i=1}^n D_{x_i} \big) \). Every operator $D_{\xi_j}$ only acts on the factor with the corresponding variable and, after the substitution $\xi_j = x_k$, every summand gives the same term $|\mathfrak{S}_{r+1}| = (r+1)!$ times, so we get: 
\begin{equation}\label{g0cnj}
(r+1)B_{0,n}\tilde{H}_{0,n}(x_{[n]})  = \!\!\!\!\! \sum_{\substack{ \{ k\} \sqcup \bigsqcup_{j=1}^{r+1} K_j = [n] }} \!   \bigg[\prod_{j = 1}^{r+1} D_{x_k} \tilde{H}_{0,|K_j| + 1}(x_{K_j},x_k)\bigg]\,.
\end{equation}
We are now ready to state and prove the following theorem.
\begin{theorem}\label{thm:g0tr}
The differentials $(\omega_{0,n})_{n \geq 3}$ satisfy the restriction to genus-zero sector of the topological recursion on \( \mc{C} \).
\end{theorem}
\begin{proof}
The strategy of the proof is analogous to the proof of \cref{thm:r3}. Indeed, \cite{KLPS17} shows that \cref{KeyP} holds for arbitrary $r$ and $q$.  As explained there, it suffices to prove the quadratic loop equation. In genus zero the quadratic loop equation for $n+1$ simplifies to the statement that the function
\begin{equation}
\label{Quale} E([n], x) \coloneqq \Delta_{z} \Delta_{z'} D_{x} D_{x'} \sum_{I \sqcup J = [n]} \tilde{H}_{0,|I|+1}(x(z),x_I) \tilde{H}_{0,|J|+1}(x(z'),x_J) \big|_{z' =z}
\end{equation}
is holomorphic in \( z\) near the ramification points of $x$, for \( 2g-2+n >0\). As before, we fix a ramification point $\rho_i$, and $\mathcal{S}$ and $\Delta$ denote the symmetrization and skew-symmetrization operators around $\rho_i$ introduced in \cref{SymASymOp}.
We argue by induction on the Euler characteristic of the factors on which the \( \Delta \)-operators act. Since the genus is equal to zero, this is an induction on $n$. Let us assume that the quadratic loop equations have been proved for all \( n' < n-1 \) and let us prove the quadratic loop equation for \(n-1\).\par
Let us apply the operator $\mc{S}_{z_1}$ to both sides of \cref{g0cnj}. The left-hand side is again holomorphic, and so are all the terms in the $k$-sums in the right-hand side, except possibly for $k=1$. Therefore the function obtained by the action of $\mathcal{S}_{z_1}$ on 
\begin{equation}\label{g0k1}
\sum_{\substack{ \bigsqcup_{j=1}^{r+1} K_j = [n]\setminus \{1\} }} \!   \bigg[\prod_{j = 1}^{r+1} D_{x_1} \tilde{H}_{0, |K_j| + 1}(x_{K_j},x_1)\bigg] =: f(x_1, \dots, x_1)
\end{equation}
should result in a holomorphic function in $z_1$. The action of $\mc{S}_{z_1}$ can be written as

\begin{equation}
\label{g0k2} \mc{S}_{z_1} f(\underbrace{x_1,\dotsc, x_1}_{r+1 \text{ times}}) 
 = 2^{-r} \!\!\! \sum_{\substack{I \sqcup J = [r+1] \\ |J|\,\,\text{even}}} \Big( \prod_{i \in I} \mc{S}_{z_1^{(i)}} \Big) \Big( \prod_{j \in J} \Delta_{z_1^{(j)}} \Big)f(x_1^{(1)},\dotsc, x^{(r+1)}_1) \Big|_{z_1^{(i)} = z_1}\,,
 \end{equation}
where we keep using the convention $x_i = x(z_i)$ and $x_1^{(i)} = x(z_1^{(i)})$ also for the new variables $x_1^{(i)}$ to shorten the notation. Let us examine the action of the different summands of the operator in the expansion above. For $J = \emptyset $, the summands produced by the action of $\prod_{i =1}^{r+1} \mc{S}_{z_i}$ are holomorphic by the linear loop equation. The first term that can possibly create non-holomorphic terms is for $|J|=2$. In that case, up to re-labeling the variables (which does not change the result since $f$ is symmetric), the term we get after the substitution $x_1^{(i)} = x_1$ reads
\begin{equation}
\begin{split}
&\sum_{\substack{ \bigsqcup_{j=1}^{r-1} K_j \sqcup \overline{K}= [n]\setminus \{1\} }}\!  
\bigg[
 \prod_{j =1}^{r-1} \mc{S}_{z_1} D_{x_1} \tilde{H}_{0, |K_j| + 1}(x_{K_j},x_1)
 \bigg] 
 \\
 &\qquad \qquad \qquad \times
  \Delta_{z_1^{(r)}} \Delta_{z_1^{(r+1)}} D_{x_1^{(r)}} D_{x_1^{(r+1)}}  \!\!\!\!\!\! \sum_{\substack{ K_r \sqcup K_{r+1} =\overline{K}}}\! 
\tilde{H}_{0, |K_{r}| + 1}(x_{K_{r}},x_1^{(r)})  \tilde{H}_{0, |K_{r+1}| + 1}(x_{K_{r+1}},x_1^{(r+1)})\big|_{z_1^{(r)} = z_1^{(r+1)} = z_1}\,.
\end{split}
\end{equation}
The first $r-1$ factors are holomorphic by the linear loop equation, whereas the second summation is holomorphic by induction hypothesis, with the exception of the one case $\overline{K} = [n] \setminus \{1\}$.  In that case we obtain the term
\begin{equation}
\left(\mc{S}_{z_1} y_1 \right)^{r-1}E\left([n] \setminus \{1\}, x_1\right)\,.
\end{equation}
since for $K_j$ empty we have
\begin{equation}
D_{x_1^{(j)}} \tilde{H}_{0, |K_j| + 1}(x_{K_j},x_1^{(j)}) = D_{x_1^{(j)}}\tilde{H}_{0,1}(x_1^{(j)}) = D_{x_1^{(j)}}H_{0,1}(x_1^{(j)})  = y_1^{(j)}\,.
\end{equation}

We remark that $(\mathcal{S}_{z_1}y_1)^{r - 1}$ is invertible near $\rho_i$ due to \cref{Lkng}. In order to deal with the terms for $|J| > 2$, we use the following lemma, whose proof is given at the end.

\begin{lemma}\label{lem:topdelta} For any $t \geq 2$, we have
\begin{equation}
\Delta^{z_1^{(1)}}\cdots \Delta^{z_1^{(2t)}}  \!\!\!\!\!\!
\sum_{\substack{ \bigsqcup_{j=1}^{2t} I_j = [n]\setminus \{1\} }} \!   \bigg[\prod_{j = 1}^{2t} D_{x_1^{(j)}} \tilde{H}_{0, |I_j| + 1}(x_{I_j},x_1^{(j)})\bigg]\bigg{|}_{x_1^{(i)} = x_1} = 
\sum_{K_1 \sqcup \dots \sqcup K_t = [n] \setminus \{1\}} \prod_{j=1}^t E(K_j, x_1).
\end{equation}
\end{lemma}

According to \cref{lem:topdelta}, a term with $|J| > 2$ in \cref{g0k1} expanded with help of \cref{g0k2} factorizes in $t = |J|/2 > 1$ quadratic loop equations multiplied by $ (r+1) - |J|$ factors of the form $\mc{S}_{z_1}D_{x_1}\tilde{H}_{0,|K_i|+1}$, which are holomorphic in $z_1$ thanks to the linear loop equation. As before, by the inductive hypothesis every quadratic loop equation factor is holomorphic, except for the one case in which one of the sets $K_i$ is equal to the whole set $[n] \setminus \{1\}$. In that case the obtained term is of the form
\begin{equation}
(\mc{S}_{z_1}y_1)^{r+1 - |J|}(\Delta_{z_1}y_1)^{|J|-2} E([n] \setminus \{1\}, x_1)\,.
\end{equation}
Collecting all the terms in which $E([n]\setminus \{1\}, x_1)$ appears we obtain the equation
\begin{equation}
E([n]\setminus \{1\}, x_1) \bigg[ \binom{r+1}{2} \left( \mc{S}_{z_1}y_1\right)^{r-1} + \binom{r+1}{2,2} \left(\mc{S}_{z_1}y_1\right)^{r-3}(\Delta_{z_1}y_1)^{2} + \cdots \bigg] = \textrm{ holomorphic in \( z_1\).}
\end{equation}
In local the coordinate $\eta$ around the ramification point $\rho_i$, we have $\Delta_{z_1}y_1 = \mc{O}(\eta)$, and so is $(\Delta_{z_1} z_1)^{2l}$ for $ l>0$. Therefore, using \cref{Lkng} the factor that multiplies $E([n]\setminus \{1\}, x_1)$ has a non-zero limit $2^{r - 2}\tfrac{r + 1}{q} (qr)^{-\frac{1}{r}}$ when $z_1 \rightarrow \rho_i$, which comes only from the first term. This factor is thus is invertible with respect to multiplication. This proves the quadratic loop equation expression for \(n-1\) is holomorphic, and hence by induction this holds for every $n \geq 1$. This concludes the proof of \cref{thm:g0tr}.
\end{proof}

\begin{proof}[Proof of \cref{lem:topdelta}]
We will prove the statement by computing the multiplicity with which a generic summand appears in the left and in the right-hand side. The fact that these two multiplicities coincide is equivalent to a simple combinatorial identity that we prove in the second part.

Let us consider first the case of the summand with an even amount of $H_{0,1}$ factors:
\begin{equation}
\big(\Delta_{z_1} D_{x_1} H_{0,1}(x_1) \big)^{2p} \Delta D_{x_1} \tilde{H}_{0,|I_1|+1}(x_{I_1}, x_1) \dots \Delta_{z_1} D_{x_1} \tilde{H}_{0,|I_{2l}|+1}(x_1, x_{I_{2l}})\,.
\end{equation}

Computing the multiplicity with which this summand appears in the left-hand side is straightforward. There are $2t$ ways to assign the set $I_1$ to a factor, $2t-1$ ways to assign the set $I_2$ and so forth up to $I_{2l}$, hence the multiplicity amounts to \( \frac{(2t)!}{(2p)!}\). Let us now work out the combinatorics for the right-hand side. Let $v$ be the number of empty sets $K_j$. We have to consider the cases $v = 0, \dots, p$ and sum up their contributions. Let us select the $v$ sets $K_j$ which are empty, this can be done in $\binom{t}{v}$ ways. Among the remaining $t-v$ sets, we have to select which are responsible for the appearance of one empty and one non-empty set in their corresponding quadratic loop equation \eqref{Quale}. Since every empty set that is not yet paired with another empty set must be paired with a non-empty set, this can be done in $\binom{t-v}{2(p-v)}$ ways. We select $2(p-v)$ non-empty sets that have to be paired with the empty ones in $\binom{2l}{2(p-v)}$ ways, and multiply by the number of bijections $\left(2(p-v)\right)!$. Let now the multinomial coefficient $\binom{2l - 2(p-v)}{2, \dots, 2}$ account for all possible pairs of non-empty sets with other non-empty sets. Finally, we multiply by a factor of $2$ for each pair that involves at least one non-empty set $I_i$. Hence we get the quantity
\begin{equation}\label{countlhs}
\sum_{v=0}^p \binom{t}{v} \binom{t-v}{2(p-v)}\binom{2l}{2(p-v)} \left(2(p-v)\right)! \binom{2l - 2(p-v)}{2, \dots, 2} 2^{t-v}\,.
\end{equation}
By simplifying the binomial coefficients, setting $m= p-v$, and dividing both sides by $(2l)!$, we see that the two multiplicities coincide if and only if the following equality is satisfied:
\begin{equation}
\sum_{m=0}^{\infty}  \binom{p+l}{p-m, l-m,2m}2^{2m} = \binom{2p + 2l}{2p}\,.
\end{equation}
with the convention that the multinomial coefficient vanishes whenever one argument in its factorials is negative.
In order to prove this equality, let us consider and rearrange the following bivariate generating series
\begin{align}
\sum_{p,l=0}^{\infty} \sum_{m=0}^{\infty} \binom{p+l}{p-m, l-m, 2m}2^{2m} X^{2p} Y^{2l}  =& \sum_{m=0}^{\infty} (2XY)^{2m} \sum_{p',l'=0}^{\infty}  \binom{p'+l' + 2m}{p', l', 2m} X^{2p'} Y^{2l'}\\
=& \sum_{m=0}^{\infty} (2XY)^{2m} \sum_{q=0}^{\infty}  \binom{q + 2m}{q} \sum_{i=0}^{q} \binom{q}{i} X^{2i} Y^{2(q-i)}\,.
\end{align}
By Newton's formula, this becomes
\begin{align}
\sum_{m=0}^{\infty} (2XY)^{2m} \sum_{q=0}^{\infty}  \binom{q + 2m}{q} (X^2 +Y^2)^q =& \left(\frac{1}{1-(X^2 + Y^2)} \right)\sum_{m=0}^{\infty} \left( \frac{2XY}{1-(X^2 + Y^2)}\right)^{2m}\\ 
=&\, \frac{1 - (X^2 + Y^2)}{(X^2 + Y^2 - 1 - 2XY - 1)(X^2 + Y^2 - 1 + 2XY)} \\
=&\,\frac{1}{2}\left[ \frac{1}{1 - (X^2 + Y^2) - 2XY} + \frac{1}{1 - (X^2 + Y^2 ) + 2XY}\right] \\
=&\,\left[\sum_{m=0}^\infty(X+Y)^{2m} \right]^{\text{even in }Y} = \sum_{p,l=0}^{\infty} \binom{2l + 2p}{2p} X^{2p} Y^{2l}\,.
\end{align}
Extracting then the coefficient $X^{2p}Y^{2l}$ from the first and the last term yields the desired equality. In case the amount of $H_{0,1}$ factors is odd (say, $2p+1$), it is enough to prove
\begin{equation}
\sum_{m=0}^{\infty}  \binom{p+l}{p-m, l-m-1,2m+1}2^{2m+1} = \binom{2p + 2l}{2p+1}\,,
\end{equation}
which can be done in the same way as above by setting $p' = p-m$ and $l' = l - m -1$.
This concludes the proof of the lemma.
\end{proof}

\newcommand{\etalchar}[1]{$^{#1}$}

\end{document}